\documentclass{amsart}
\usepackage{amssymb,amsfonts,latexsym}

\newtheorem{theorem}{Theorem}[section]
\newtheorem{lemma}[theorem]{Lemma}
\newtheorem{proposition}[theorem]{Proposition}
\newtheorem{corollary}[theorem]{Corollary}
\newtheorem{claim}[theorem]{Claim}
\theoremstyle{definition}
\newtheorem{definition}[theorem]{Definition}

\newtheorem{example}[theorem]{Example}

\newtheorem{remark}[theorem]{Remark}
\newtheorem{general remarks}[theorem]{General remarks}
\newtheorem{note}[theorem]{Note}

\newtheorem{terminology}[theorem]{Terminology}
\newtheorem{explanation}[theorem]{Explanation}

\newcommand{\Id}{\operatorname{Id}}

\newcommand{\Par}{\operatorname{Par}}

\newcommand{\ben}{\begin{enumerate}}
\newcommand{\een}{\end{enumerate}}

\begin{document}

\title[Kemer's theorem for affine PI algebras over a field of characteristic zero]
{Kemer's theorem for affine PI algebras over a field of
characteristic zero}

\author{Eli Aljadeff}
\address{Department of Mathematics, Technion-Israel Institute of
Technology, Haifa 32000, Israel}
\email{aljadeff@tx.technion.ac.il}

\author{Alexei Kanel-Belov}
\address{Department of Mathematics, Bar-Ilan University, Ramat-Gan, Israel}
\email{kanel@mccme.ru}

\author{Yaakov Karasik}
\address{Department of Mathematics, Technion-Israel Institute of
Technology, Haifa 32000, Israel}
\email{yaakov@tx.technion.ac.il}
\date{Feb. 15, 2015}






\thanks {The first author was partially supported by the ISRAEL SCIENCE FOUNDATION
(grant No. 1017/12) and by the E.SCHAVER RESEARCH FUND. The second
author was partially supported by the ISRAEL SCIENCE FOUNDATION
(grant No. 1178/06). The second author is grateful to the Russian Fund of Fundamental
Research for supporting his visit to India in 2008 (grant  $\ RFBR 08-01-91300-IND_a$).}

\begin{abstract} We present a proof of Kemer's representability theorem for affine PI algebras over a field of characteristic zero.

\end{abstract}

\maketitle

\maketitle

\begin{section}{Introduction} \label{Introduction}

The purpose of this article is to provide a proof, as complete as
possible, of the representability theorem for affine PI algebras
over a field of characteristic zero which is due to Kemer.

\begin{theorem}

Let $W$ be an affine PI algebra over a field $F$ of characteristic
zero. Then there exists a field extension $L$ of $F$ and a finite
dimensional algebra $A$ over $L$ which is PI equivalent to $W$
$($i.e. $\Id_{L}(W_{L})=\Id_{L}(A)$, where $W_{L}=W\otimes_{F}L$$)$.
\end{theorem}
\end{section}

Kemer's proof may be found in his original article \cite{Kemer1.5} and in his monograph \cite{Kemer2}.
Our exposition here follows the steps of the proof as in
\cite{Al-Bel}. The paper \cite{Al-Bel} provides a generalization
of Kemer's Theorem for PI $G$-graded algebras $W$ over a field of
zero characteristic and $G$ an arbitrary finite group. As had been
mentioned in \cite{Al-Bel}, the proof there follows (at least
partially) the general idea of the proof in the ungraded case
which appears in \cite{BR}. 

Besides \cite{Kemer1.5}, \cite{Kemer2}, \cite{BR} and \cite{Al-Bel}, let us also mention
\cite{Sviridova1} and \cite{Sviridova2} which provide a generalization of Kemer's
theorem for PI $G$-graded algebras where $G$ is finite abelian and for affine PI
algebras with involution.

\begin{section}{Sketch of the proof} \label{Sketch of the proof}

In this short section we outline the main steps of the proof of Kemer's theorem.

\begin{enumerate}

\item
Step $1$. Show there exists a finite dimensional algebra $A$ with
$\Id(A) \subseteq \Gamma= \Id(W)$. This is a highly nontrivial result which
uses Rasmyslov, Kemer, Braun theorem on the nilpotency of the Jacobson Radical of an affine PI algebra, Kaplansky' theorem on primitive PI algebras and Lewin's theorem on products of $T$-ideals.

\item
Step $2$. Definition of $Ind(\Gamma)$, the Kemer index of any
$T$-ideal $\Gamma$ which contains the $T$-ideal of a finite
dimensional algebra $A$. In case $\Gamma = \Id(A)$, where $A$ is
a finite dimensional algebra, we refer to the index of $\Id(A)$ as the index
of $A$ and denote it by  $Ind(A)$. The Kemer index of $\Gamma$
consists of two parameters (nonnegative integers) $(\alpha, s)$,
which provide a ``measure'' of the extreme alternating properties
of polynomials which are not in $\Gamma$. It is clear that any
multilinear polynomial which is alternating on a set of
cardinality $> dim_{F}(A)$ must be in $\Id(A)$ and hence in
$\Gamma$. It follows that the cardinality of alternating sets of
any multilinear polynomial $f$ \textit{which is not} in $\Gamma$
is uniformly bounded (e.g. by $dim_{F}(A)$). Strictly speaking non
of the parameters $\alpha$ or $s$ ( where $Ind(\Gamma)=(\alpha,s)$) measures the
maximal cardinality of an alternating set in a polynomial $f
\notin \Gamma$. They measure more subtle invariants related to
alternation in polynomials $f \notin \Gamma$.

As mentioned above, the Kemer index $Ind(\Gamma)=(\alpha,s)$ is an
element in $\Omega=\mathbb{Z}^{\ge0}\times \mathbb{Z}^{\ge0}$. We
consider the lexicographic ordering  on $\Omega$ and denote it by
$\leq$. From the definition of $Ind(\Gamma)$ it will be easy to
conclude that if $\Gamma_{1} \subseteq \Gamma_{2}$ then
$Ind(\Gamma_{2}) \leq Ind(\Gamma_{1})$ (reverse ordering). In
particular $Ind(\Gamma) \leq Ind(A)$.

\item
Step $3$. Definition of Kemer polynomials of a $T$-ideal $\Gamma$.
These are extremal polynomials which are not in $\Gamma$ whose
alternation realize the Kemer index $Ind(\Gamma)$.

\item

Step $4$. Construction of basic algebras. These are finite
dimensional algebras that on one hand they are ``abundant'' enough
so that any finite dimensional algebra is PI equivalent to the
direct product of finitely many basic algebras and on the other
hand they are ``special'' enough so that the parameters $\alpha$
and $s$ (of $Ind(A)$) coincide respectively with the integers
$\dim_{F}(\overline{A})$ and $n_{A}-1$, where $\overline{A}$ is
the semisimple subalgebra of $A$ which supplements the radical
$J(A)$ and $n_{A}$ is the nilpotency index of $J(A)$.

\item

Step $5$. From the connection between the parameters of the Kemer
index of any basic algebra $A$ and its geometrical properties (namely
$\dim_{F}(\overline{A})$ and $n_{A}-1$) we obtain the Phoenix
property of Kemer polynomials of $A$. By definition, Kemer
polynomials satisfy the Phoenix property if for any Kemer polynomial $f$ of $\Gamma$ (in particular $f \notin \Gamma$) and any
polynomial $f' \in \langle f \rangle$ which is not in $\Gamma$,
there is $f'' \in \langle f' \rangle$ which is Kemer of $\Gamma$
(i.e. Kemer polynomials regenerate themselves). By the fact that
any finite dimensional algebra $A$ is PI equivalent to the product
of some basic algebra we conclude the Phoenix property of Kemer
polynomials of finite dimensional algebras.

\item
Step $6$. Find a finite dimensional algebra $B$ with $\Id(A)
\subseteq \Id(B) \subseteq \Gamma$ so that $\Id(B)$ and $\Gamma$
have the same Kemer index and have the same Kemer polynomials.
From that we obtain the Phoenix property of any Kemer polynomial
of $\Gamma$.

\item

Step $7$. Construction of a representable algebra $B_{(\alpha,s)}$ over $F$ (i.e. contained in the algebra of $n \times n$-matrices over $L$ for some $n$, and $L$ is a field extension of $F$)  with
$\Id(B_{(\alpha,s)}) \supseteq  \Gamma$ and such that all Kemer
polynomials of $\Gamma$ are nonidentities of $B_{(\alpha,s)}$.

\item

Step $8$. We consider $\Gamma'= \Gamma + S $ where $S$ is the
$T$-ideal \textit{generated} by all Kemer polynomials of $\Gamma$.
This will imply that $Ind(\Gamma') < Ind(\Gamma)$ and hence by
induction on the Kemer index there exists a finite dimensional
algebra $A'$ (over a field extension $L$ of $F$) with
$\Gamma'=\Id(A')$.

\item

Step $9$. We show that \textit{all} polynomials of $S$ (which are
not in $\Gamma$) are nonidentities of $B_{(\alpha,s)}$ (that is,
not just elements in $S$ which are Kemer polynomials as shown in
Step $7$). This is achieved by the Phoenix property of Kemer
polynomials. Since any nonidentity $f'$ of $\Gamma$ which is in
$S$, produces (by the $T$-operation) a Kemer polynomial which by
Step $7$ is not in $\Id(B_{(\alpha,s)})$ we have also that $f'
\notin \Id(B_{(\alpha,s)})$. From that one concludes that $\Gamma =
\Id(A' + B_{(\alpha,s)})$. Being the algebra $A' + B_{(\alpha,s)}$ sum of representable algebras, it is representable and so we are done.

\end{enumerate}
\end{section}

\begin{section}{Start} \label{Start}

The starting point of the Representability Theorem for affine PI algebras, is the existence of a finite dimensional $F$-algebra $B$ (here $F$ is a field of characteristic zero) whose ideal of identities $\Id_{F}(B)$ is contained in
$\Gamma$ (the given $T$-ideal of an affine PI $F$-algebra $W$). The existence of such algebra $B$ relies on three fundamental theorems of PI theory. In what follows
we prove this assertion and state (without proof) along the way the
theorems when they are needed. For future reference let us formulate the main theorem of this section.

\begin{theorem} \label{affine PI satisfies a Capelli}
Let $\Gamma$ be the $T$-ideal of identities of an affine PI
algebra $W$ over a field $F$ of characteristic zero. Then there
exists a finite dimensional
algebra $B$ over $F$ such that $\Id(B) \subseteq \Gamma$.

\end{theorem}

Consider the relatively free algebra $\mathcal{W}=F\left\langle X\right\rangle /\Gamma$.
Denote by $\mathcal{J}$ its Jacobson radical and by $J$ the pre-image
of $\mathcal{J}$ inside $F\left\langle X\right\rangle $. It is clear
that $J$ contains $\Gamma$.

First, we claim that $J$ is a $T$-ideal (an ideal of $F\langle X \rangle$ is a $T$-ideal if it is invariant under any endomorphism of the $F$-algebra $F\langle X \rangle$). For this recall that $r\in J(C)$ (the Jacobson radical of an $F$-algebra $C$) if and only if any element of $rC$ is right-quasi-invertible
(i.e. for every $x\in rC$ there is $y_{x}\in C$ (quasi-inverse of $x$) such that $xy_{x}-x-y_{x}=0$).
Therefore, $J(C)$ is the unique maximal right-quasi-regular right-ideal
of $C$ (a right ideal $U$ of $C$ is right-quasi-regular if every element $z \in U$ is quasi-invertible). It follows that if $I\triangleleft\mathcal{W}$ is a right-quasi-regular
right-ideal, then so is the image of $I$ (inside $\mathcal{W}$) by any surjective endomorphism of $\mathcal{W}$. Hence, by the maximality of $\mathcal{J}$, we deduce that the image of $\mathcal{J}$ under any epimorphism  $\mathcal{W}\rightarrow \mathcal{W}$
is contained in $\mathcal{J}$. We conclude that the
image of $J$ under any surjective endomorphism of $F\left\langle X\right\rangle $
is contained $J$. We have proved the invariance of the ideal $J$ under endomorphisms of $F\left\langle X\right\rangle $ which are epic. In order to show the invariance of $J$ under arbitrary endomorphisms of $F\left\langle X\right\rangle $, let $\phi\in End_{F}F\left\langle X\right\rangle $, and define
$\phi_{n}\in End_{F}F\left\langle X\right\rangle $ by
\begin{eqnarray*}
\phi_{n}(x_{i}) & = & \phi(x_{i}),\,\, i=1,2,...,n\\
\phi_{n}(x_{i}) & = & x_{i-n},\,\, i=n+1,n+2,...
\end{eqnarray*}
Note that $\phi_{n}$ is onto $F\langle X \rangle$ and $\phi_{n}|_{F\left\langle x_{1},...,x_{n}\right\rangle }=\phi|_{F\left\langle x_{1},...,x_{n}\right\rangle }$.
Thus, $\phi\left(J\cap F\left\langle x_{1},...,x_{n}\right\rangle \right)\subseteq J$.
This holds for every $n$, so the claim follows.

Next, consider the algebra $\mathcal{W}^{\prime}=\mathcal{W}/\mathcal{J}=F\left\langle X\right\rangle /J$.
This is the relatively free algebra associated with $J$. Since clearly
the Jacobson radical of $\mathcal{W}^{\prime}$ is zero, this algebra
is a sub-direct product of $A_{\mbox{Ind}}=\prod_{i\in\mbox{Ind}}A_{i}$,
where $\mbox{Ind}$ is some index set and each $\ensuremath{A_{i}}$
is a primitive $F$-algebra. Note that all the $A_{i}$ satisfy the
identities of $\mathcal{W}$ and $\Id_{F}(A_{\mbox{Ind}})=\Id_{F}(\mathcal{W}^{\prime})$.
\begin{theorem}
[Kaplansky]Suppose $A$ is a primitive $F$-algebra satisfying a
PI of degree $d$. Then $A$ is a central simple algebra and $\mbox{dim}_{\mbox{Cent}(A)}A\leq\left\lfloor d/2\right\rfloor $.
\end{theorem}
Hence, $\mathcal{W}^{\prime}$ has the same identities as $\prod_{i\in\mbox{Ind}}M_{n_{i}}(F_{i})$
, where $F_{i}$ are fields containing $F$ and all the $n_{i}$ are
uniformly bounded. The ideal of identities of this algebra is equal
to $\cap_{i\in\mbox{Ind}}\Id_{F}(M_{n_{i}}(F))$ ($\Id_{F}(M_{n_{i}}(F_i))=\Id_{F}(M_{n_{i}}(F)))$, thus there is some
$n_{0}$ such that this ideal is equal to $\Id_{F}(M_{n_{0}}(F))$.
We conclude that $\Id_{F}(M_{n_{0}}(F))=\Id_{F}(\mathcal{W}^{\prime})=J$.

In order to continue we need to know that $\mathcal{J}$ is nilpotent.
\begin{theorem}
[RKB]\label{Nilp_Jacobson}If $A$ is an affine PI $F$-algebra, then its Jacobson radical
is nilpotent.
\end{theorem}
Note that we cannot apply Theorem \ref{Nilp_Jacobson} directly to $\mathcal{W}$ since (in general) it is
nonaffine. However, since an element $c$ of an algebra $C$ is in
$J(C)$ if and only if $cC$ consists of right-quasi-regular elements,
we deduce that the evaluations of $J$ on $W$ are inside $J(W)$.
Since $W$ is affine, there is an integer $r$ for which $J(W)^{r}=0$,
thus $J^{r}\subseteq \Id_{F}(W)=\Gamma$.

We complete the proof by showing that $J^{r}$ is the $T$-ideal of a certain finite dimensional algebra over $F$. Recall that $J$ is the $T$-ideal
of identities of $M_{n_{0}}(F)$.

\begin{theorem}
[Lewin]Let $B=UT_{F}(d_{1},...,d_{m})$ be the upper block triangular
matrix algebra over $F$. Then, $\Id_{F}(B)=\Id_{F}(M_{d_{1}}(F)\cdots \Id_{F}(M_{d_{m}}(F))$.
\end{theorem}
Therefore, $J^{r}$ is the ideal of identities of the finite dimensional
$F$-algebra $UT_{F}(\underset{r\mbox{ times}}{\underbrace{n_{0},...,n_{0}}})$.

It follows that there exists an $n$ such that any
polynomial which alternates on a set of cardinality $n$ is an
identity of $W$. In that case we say that $c_{n}$, the $n$th Capelli polynomial, is in $\Gamma$ or that $W$ satisfies the $n$th Capelli identity.

\begin{remark}
It is well known (and easy to prove) that any $T$-ideal over a field of characteristic zero is generated by multilinear polynomials and hence $T$-ideals are stable under field extensions. This means that if $\Gamma = \Id(W)$ where $W$ is an algebra over $F$, then $\Gamma \otimes_{F}L=\Id_{L}(W \otimes_{F}L)$. It follows that we may assume, by extensions of scalars, that the algebra $A$ (appearing in \ref{affine PI satisfies a Capelli}) is finite dimensional over $F$ (rather than $K$). In addition, we assume as we may that $F$ is algebraically closed.

\end{remark}

\end{section}
\begin{section}{the index of $T$-ideals}\label{the index of $T$-ideals}

The following terminology will be used frequently.

\begin{definition}
Let $f(X,Y)$ be a polynomial in noncommuting variables where $X=\{x_1,\ldots,x_{n}\}$ and $Y$ is an arbitrary finite set. Suppose further that $f$ is multilinear on the set $X$. We say that $f$ is alternating on $X$ if there exists a polynomial $h=h(X,Y)$ (multilinear on the set $X$) such that
$$
f(x_1,\ldots,x_{n},Y)=\sum_{\sigma \in Sym(n)}(-1)^{\sigma}h(x_{\sigma(1)},\ldots,x_{\sigma(n)},Y).
$$
\end{definition}

\begin{terminology}\textit{Multialternating polynomial} We consider polynomials with $\nu$ disjoint
sets of alternating variables of cardinality $d$ and $\mu$ disjoint sets (also disjoint
to the previous sets) of alternating variables of cardinality $l$. In particular we consider the case $l=d+1$.

\end{terminology}

Given a $T$-ideal of identities $\Gamma$ of an affine PI algebra
$W$ we have by the previous section that there exists an integer $k$ such that any polynomial
$f$ with an alternating set of $k$ variables is necessarily in
$\Gamma$ (indeed, any $k > dim_{F}(A)$ will do). So if $d$ is an integer for which there
exist polynomials with alternating sets of cardinality $d$ which are not in $\Gamma$,
then clearly $d < k$. Fix a nonnegative integer $\nu$ and
consider the largest integer $d_{\nu}$ for which there exist
multilinear polynomials \textbf{not} in $\Gamma$  with $\nu$
disjoint alternating sets of cardinality $d_{\nu}$. Clearly, the
function $d_{\nu}$ is nonnegative, integer valued and
nonincreasing. Denote by $d= \lim d_{\nu}$.
Furthermore, we let $\mu$ be the integer $\nu$ for which the
function $d_{\nu}$ stabilizes (i.e. reaches the limit). Note that
$d$ may be zero as it is in the case of nilpotent algebras.

By the
definition of the integer $d$ it follows that there exist multilinear
polynomials not in $\Gamma$ with arbitrary many disjoint
alternating sets of variables of cardinality $d$. On the other hand (again from the definition of $d$)
there is a bound on the number of alternating
disjoint sets of cardinality $d+1$ that one can find in
multilinear polynomials outside $\Gamma$. Let us denote by
$s_{\nu}$ the maximal number of alternating sets of cardinality
$d+1$ which appear in polynomials not in $\Gamma$ which have $\nu$
alternating sets of cardinality $d$. Again, the function $s_{\nu}$
is nonincreasing, nonnegative integer valued and
hence has a limit which we denote by $s$. So we have constructed
three integers, namely $d$, $\mu$ and $s$ (we take $\mu$, to be
large enough so that the two functions $d_{\nu}$ and $s_{\nu}$
reach the limit). By construction we can find multilinear
polynomials outside $\Gamma$ with arbitrary many alternating
disjoint sets of cardinality $d$ and precisely $s$ sets of
alternating sets of cardinality $d+1$. Furthermore, as long as the
number of alternating sets of cardinality $d$ exceeds (or equal)
$\mu$ we will not find multilinear polynomials outside $\Gamma$
with more than $s$ alternating sets of cardinality $d+1$.

\begin{example}

Let $W$ be a PI algebra and denote by $\Gamma=\Id(W)$ its
$T$-ideal of identities. Then $W$ is nilpotent with
\textit{nilpotency index} $=n$ ($W^{n-1}\neq 0$ and $W^{n}=0$) if
and only if $d(\Gamma)=0$ and $s(\Gamma)=n-1$.
\end{example}

\begin{proof}
Note that $W$ is a nilpotent algebra with nilpotency index $n
\geq 1$ if and only if $x_1\cdots x_n \in \Gamma$, and $n$ is
minimal. It follows there are no multilinear polynomials outside
$\Gamma$ with arbitrary many alternating sets of cardinality $1$
and hence $d=0$. On the other hand we have at most $n-1$
(alternating sets) of cardinality $1$ (in a nonidentity) and hence
$s\leq n-1$. Since $x_1\cdots x_{n-1} \notin \Gamma$,
$s(\Gamma)=n-1$. Conversely, if $d=0$, there is a bound on the
number of alternating sets of cardinality $1$ and hence there is a
bound on the length of multilinear monomials. If $m$ is the
largest length, it is clear that the nilpotency index is $m+1$ as
required.
\end{proof}

\begin{remark}
Note that if $W$ is an algebra over $F$, affine and nilpotent, then it is clearly finite dimensional over $F$ and hence representable. This allows us to focus (in certain proofs) on affine PI algebras with index $(\alpha, \gamma)$ where $\alpha \geq 1$.
\end{remark}

\begin{definition}
We say that a multilinear polynomial $f$ is a Kemer polynomial of
$\Gamma$ if (a) $f$ is not in $\Gamma$ (b) $f$ has at least $\mu$
alternating disjoint sets of cardinality $d$ (\textit{small sets}) (c) $f$ has exactly
$s$ alternating disjoint sets of cardinality $d+1$ (\textit{big sets}).

\end{definition}

Given a $T$-ideal of identities of an affine PI algebra we extract
two parameters, namely $d$ and $s$ and consider the pair
$Ind(\Gamma)=(d, s)$ (Kemer index of $\Gamma$) as a point in the
set $\Omega=Z^{\geq 0} \times Z^{\geq 0}$ with the (left)
lexicographic ordering (denoted by $\leq$). So any $T$-ideal of
identities of an affine PI algebra determines a point in $\Omega$.

\begin{remark}
If $W$ is non-PI, we have $\Gamma=0$ and $d$ is not bounded. We
are assuming this is not the case. If $W=0$, the $T$-ideal is the
free algebra and hence $Ind(\Gamma)=(0,0)$. If $W$ is an affine PI
algebra and nonzero, $Ind(\Gamma)> (0,0)$ and is in $Z^{\geq 0}
\times Z^{\geq 0}$. If $W$ is PI but nonaffine, then it may
satisfy a Capelli identity and it may not. If it satisfies a
Capelli identity then it has an index in $\Omega$ and if it
doesn't satisfy any Capelli identity it does not have an index in
$\Omega$. Note that by Theorem \ref{affine PI satisfies a Capelli}
any affine PI algebra satisfies a Capelli identity $c_{n}$ for
some $n$.

\end{remark}

\begin{lemma}\label{comparison}

If $\Gamma_{1} \subseteq \Gamma_{2}$ then $Ind(\Gamma_{1}) \geq
Ind(\Gamma_{2})$ i.e. the order is reversed.
\end{lemma}

Before proving the lemma (below) let us explain once again (in more formal terminology) the functions $d_{\nu}$ and $s_{\nu}$ considered above.

\begin{explanation} \textit{{The function $d_{\nu}$, $s_{\nu}$}}

We are assuming that there exists a finite dimensional algebra $A$
(say of dimension $n$) such that $\Id(A) \subseteq \Gamma$. The
algebra $A$ satisfies $c_{n+1}$ meaning that any multilinear
polynomial $p(X)$ which is alternating on a set of cardinality
$n+1$ is in $\Gamma$. In other words, outside $\Gamma$ there is no
multilinear polynomial which is alternating on a set of
cardinality $n+1$.

Fix an integer $\nu=1,2,\ldots,$. Let $d_{\nu}$ be the maximal
nonnegative integer such that there exists \textit{outside}
$\Gamma$ a multilinear polynomial which is alternating on $\nu$
disjoint sets of cardinality $d_{\nu}$. Note that $d_{\nu} < n+1$
since any polynomial with at least one alternating set of
cardinality $n+1$ is already in $\Gamma$. This says that we can
find outside $\Gamma$ a multilinear polynomial with $\nu$
alternating (disjoint) sets of cardinality $d_{\nu}$ but any
polynomial with $\nu$ alternating sets of cardinality $d_{\nu}+1$
is already in $\Gamma$. In that sense $d_{\nu}$ is maximal.

Now take $\rho > \nu$. Then, by definition, there is a multilinear
polynomial $q(X) \notin \Gamma$ with $\rho$ alternating sets of
cardinality $d_{\rho}$. Let us show $d_{\rho} \leq d_{\nu}$.
Suppose $d_{\rho} > d_{\nu}$. This would contradict the maximality
of $d_{\nu}$ since the polynomial $q(X)$ has $\rho$ sets (and
hence $\nu$ sets) of cardinality $d_{\rho}$ which are alternating.
Consequently, the function $d: \mathbb{N}\rightarrow
\mathbb{N}\cup \{0\}$ is nonincreasing and hence has a limit which
we denote by $d$. We denote by $\mu$ the minimal integer with
$d_{\mu}=d_{\mu+1}=d_{\mu+2}=\dots$.

Here is the interpretation of $d$ and $\mu$. There exist
polynomials outside $\Gamma$ with arbitrary many alternating sets
(disjoint) of cardinality $d$. On the other hand we will not find
polynomials outside $\Gamma$ with arbitrary many alternating sets
of cardinality $d+1$. By the definition of $\mu$ above, the number
of alternating sets of cardinality $d+1$ we can find in
polynomials outside $\Gamma$ is bounded by $\mu-1$.

\end{explanation}

Let us return to the above considerations with somewhat more
formal notation and prove Lemma \ref{comparison}. For $\nu \in
\mathbb{N}$ consider the set of nonnegative integers

\bigskip

$\Delta_{\nu} = \{r \in \mathbb{N} \cup \{0\}: \exists p(X) \notin
\Gamma,$ alternating on $\nu$ disjoint sets of cardinality $r$\}.

\bigskip

The set $\Delta_{\nu}$ is bounded by $n$ (the dimension of $A$).
We denote by $d_{\nu}$ its maximum. Since any $\rho$ alternating
set contains a $\nu$ alternating set for any $\nu < \rho$ we have
that $\Delta_{\rho} \subseteq \Delta_{\nu}$ and hence $d_{\rho}
\leq d_{\nu}$. We denote by $d$ the limit of $d_{\nu}$.

\bigskip
In order to construct the second parameter of the Kemer index of
$\Gamma=\Id(W)$ ($W$ affine and PI), we know by the definition of
the parameter $d$, that for any $\nu \geq 1$, there is a
multilinear polynomial $p_{\nu}(X) \notin \Gamma$ which alternates
on $\nu$ sets of cardinality $d$. Clearly, the parameter $d$
depends on the $T$-ideal $\Gamma$ so in what follows we may write
$d_{\Gamma}$. From the definition of the parameter $d_{\Gamma}$ we
know that the set

$S^{(\Gamma)}_{\nu} = \{j \in \mathbb{N} \cup \{0\}: \exists p(X)
\notin \Gamma,$ alternating on $\nu$ disjoint sets of cardinality
$d_{\Gamma}$ and $j$ disjoint sets (and disjoint to the previous
sets) of cardinality $d_{\Gamma}+1$ \} is nonempty. Furthermore,
by the maximality of $d_{\Gamma}$, we know the set
$S^{(\Gamma)}_{\nu}$ is bounded for every $\nu$ and we denote by
$s_{\nu} = max(S^{(\Gamma)}_{\nu})$. As it is for the sequence
$\{d_{\nu}\}$, also the sequence $\{s_{\nu} \}$ is nonincreasing
and we let $s=s_{\Gamma}=lim_{\nu \rightarrow \infty} s_{\nu}$. We
let $\mu$ be an integer where the sequences $\{d_{\nu}\}$ and
$\{s_{\nu}\}$ reach the limit.

\begin{remark}

In what follows it \textit{will not be important} to keep the
precise value $\mu$ (i,e, the minimal value) where the sequences
$\{d_{\nu}\}$ and $\{s_{\nu}\}$ stabilize. We can take larger
integers.
\end{remark}

\begin{proof}{(of Lemma \ref{comparison})}

Consider the sets $\Delta^{(\Gamma_1)}_{\nu}$ and
$\Delta^{(\Gamma_2)}_{\nu}$ which correspond to integer $\nu$ and
the $T$-ideals $\Gamma_{1}$ and $\Gamma_{2}$ respectively. Since
$\Gamma_{1} \subseteq \Gamma_{2}$ we have
$\Delta^{(\Gamma_1)}_{\nu} \supseteq \Delta^{(\Gamma_2)}_{\nu}$.
Consequently, $(d_{\Gamma_1})_{\nu}=max(\Delta^{(\Gamma_1)}_{\nu})
\geq (d_{\Gamma_2})_{\nu}=max(\Delta^{(\Gamma_2)}_{\nu})$ for
every $\nu$ and hence $d_{\Gamma_1} \geq d_{\Gamma_2}$. In order
to complete the proof of the lemma, we need to show that if
$d_{\Gamma_1} = d_{\Gamma_2}$ then $s_{\Gamma_1} \geq
s_{\Gamma_2}$. To see this, note that in that case
$S^{(\Gamma_1)}_{\nu} \supseteq S^{(\Gamma_2)}_{\nu}$ for every
$\nu$ and hence $(s_{\Gamma_1})_{\nu}=max(S^{(\Gamma_1)}_{\nu})
\geq (s_{\Gamma_2})_{\nu}=max(S^{(\Gamma_2)}_{\nu})$ for every
$\nu$. Taking the limit we have $s_{\Gamma_1} \geq s_{\Gamma_2}$
and we are done.
\end{proof}

\end{section}

\begin{section}{The index of finite dimensional
algebras}\label{The index of finite dimensional
algebra}


We start this section with the definition of the Phoenix property.

\begin{definition}(The Phoenix property)
Let $\Gamma$ be a $T$-ideal as above. Let $P$ be any property
which may be satisfied by polynomials (e.g. being Kemer). We say
that $P$ is ``\textit{$\Gamma$-Phoenix}'' (or in short
``\textit{Phoenix}'') if given a multilinear polynomial $f$
satisfying $P$ which is not in $\Gamma$ and \textit{any} $f^{'}$
in $\langle f \rangle$ (the $T$-ideal generated by $f$) which is
not in $\Gamma$ as well, there exists a multilinear polynomial
$f^{''}$ in $\langle f^{'} \rangle$ which is not in $\Gamma$ and
satisfies $P$. We say that $P$ is ``\textit{strictly
$\Gamma$-Phoenix}'' if any multilinear polynomial $f^{'} \in
\langle f \rangle$ which is not in $\Gamma$, satisfies $P$.

\end{definition}

\begin{remark}

Given a polynomial $g$, there exists a multilinear polynomial
$f'$ such that $\langle f'\rangle = \langle g \rangle$. It follows
that in order to verify the Phoenix property it is sufficient to
consider multilinear polynomials $f'$  in  $\langle f \rangle$.

\end{remark}

\begin{example}
Multilinearization implies that ``multilinearity'' is Phoenix:
indeed, if $f$ is any multilinear polynomial not in $\Gamma$ and
$f' \in \langle f \rangle$ which is not in $\Gamma$ there exists a
multilinear polynomial $f'' \in \langle f' \rangle$ which is not
in $\Gamma$.

\end{example}

Let us pause for a moment and summarize what we have at this
point. We are given a $T$-ideal $\Gamma$ (the $T$-ideal of
identities of an affine algebra $W$). We assume that $W$ is
\textit{PI} and hence as shown in Section \ref{Start} there exists
a finite dimensional algebra $A$ with $\Gamma \supseteq \Id(A)$.
To the $T$-ideal $\Gamma$ we attach the corresponding Kemer index
in $Z^{\geq 0} \times Z^{\geq 0}$. Similarly, we may consider the
Kemer index of $\Id(A)$ which by abuse of notation we denote it by
$Ind(A)$. By Lemma \ref{comparison}, we have  $Ind(\Gamma) \leq
Ind(A)$.

One of our main goals (in the first part of the proof) is to
replace the algebra $A$ by a representable algebra $A^{'}$ with a larger
$T$-ideal such that

\begin{enumerate}

\item

$\Gamma \supseteq \Id(A^{'})$

\item

$\Gamma$ and $\Id(A^{'})$ have the same Kemer index.

\item

$\Gamma$ and $\Id(A^{'})$ have the ``same'' Kemer polynomials.

\end{enumerate}

\begin{remark}
The terminology ``the same Kemer polynomials'' needs a
clarification. Suppose $\Gamma_{1} \supseteq \Gamma_{2}$ are
$T$-ideals with $Ind(\Gamma_{1})=Ind(\Gamma_{2})$. We say that
$\Gamma_{1}$ and $\Gamma_{2}$ have the same Kemer polynomials if
there exists an integer $\mu$ such that all Kemer polynomials of
$\Gamma_{2}$ with at least $\mu$ alternating small sets are not in
$\Gamma_{1}$.
\end{remark}

\begin{remark}
Statements $(1)-(3)$ above will establish the important connection between the
combinatorics of the Kemer polynomials of $\Gamma$ and the
structure of finite dimensional algebras. The
``Phoenix'' property for the Kemer polynomials of $\Gamma$ will
follow from that connection.
\end{remark}

Let $A$ be a finite dimensional algebra over $F$ and let $J(A)$ be its Jacobson radical.
We know that $\overline{A}=A/J(A)$ is
semisimple. Moreover by the Wedderburn-Malcev
Principal Theorem there
exists a semisimple subalgebra $\overline{A}$
of $A$ such that $A= \overline{A} \oplus J(A)$ as
vector spaces. In addition, the subalgebra $\overline{A}$
may be decomposed as an algebra into the
direct product of (semisimple) simple algebras
$ \overline{A}\cong A_{1}\times A_{2}\times\cdots\times A_{q}$.

\begin{remark} This decomposition enables us to consider ``semisimple'' and
``radical'' substitutions. More precisely, since in order to check
whether a given multilinear polynomial is an identity of $A$ it is
sufficient to evaluate the variables on any (given) spanning set, we may
take a basis consisting of elements of $\overline{A} \cup J(A)$. We refer to such evaluations as semisimple or radical
evaluations respectively. Moreover, the semisimple substitutions
may be taken from the simple components.

\textit{In what follows, whenever
we evaluate a polynomial on a finite dimensional algebra, we
consider only evaluations of that kind.}
\end{remark}
\bigskip

For any finite dimensional algebra $A$ over $F$ we let $d(A)$ be
the dimension of the semisimple subalgebra and $n_{A}$ the
nilpotency index of $J(A)$. We denote by $Par(A)= (d(A), n_{A}-1)$
the parameter of the algebra $A$.

\begin{lemma} \label{polynomials with many alternating sets are identities}

Let $A$ be a finite dimensional algebra over $F$ and $Par(A)=
(d(A), n_{A}-1)$ its parameter. If $f$ is a multilinear polynomial
with at least $n_{A}$ alternating sets $\{X_{l}\}_{l}$, each of
cardinality $d(A)+1$, then $f$ is an identity of $A$.

\end{lemma}

\begin{proof}
If one of the sets $X_{l}$ is evaluated only with semisimple
elements, by the pigeonhole principle we must have repetitions,
and hence by the alternation the polynomial vanishes. Otherwise
all sets $X_{l}$ get at least one radical evaluation and again the
polynomial vanishes since their number is at least the nilpotency
index of $J(A)$.
\end{proof}

\begin{proposition}\label{index-par inequality}

Let $(\alpha, s)$ be the index of $A$. Then $(\alpha, s) \leq
(d(A), n_{A}-1)$.

\end{proposition}

\begin{proof}
By the definition of the parameter $\alpha$, there exist
nonidentity polynomials with arbitrary large number of alternating
sets of cardinality $\alpha$. This says that $\alpha \leq d(A)$
for if $\alpha > d(A)$, by the previous lemma we cannot have in a
nonidentity more than $n(A)-1$ alternating sets of cardinality
$\alpha$. In order to complete the proof of the proposition we
need to show (by the lexicographic ordering) that if $\alpha=d(A)$
then $s \leq n_{A}-1$. If not, by the definition of $Ind(A)$ there
exists a nonidentity of $A$ with $s > n_{A}-1$ alternating sets of
cardinality $\alpha+1$ ($=d(A)+1$) which is again impossible by
the previous lemma. This proves the proposition.

\end{proof}

In the next example we show that the index of $A$ may be
quite far from $\Par(A)=(d(A), n_{A}-1)$.

\begin{example}

Let B be a finite dimensional simple algebra, $r > 1$,  $B(r)= B \times\cdots\times B$  ($r$ times). Clearly $\Id(B)=\Id(B(r))$
and hence $B$ and $B(r)$ have the same Kemer index. On the other hand $Par(B(r))$ increases with $r$:
$Par (B)=(d(B),1)< (r\cdot d(B),1)=Par(B(r))$.

\end{example}

In order to establish a precise
relation between the index of a finite dimensional
algebra $A$ and its structure we need to find appropriate finite
dimensional algebras which will serve as \underline{minimal
models} for a given Kemer index.

\begin{definition}
Let $A$ be a finite dimensional algebra over $F$. We say $A$ is
\textit{basic} if $A$ is not PI equivalent to an algebra $B$ where $B=B_1
\times \cdots \times B_r$, $B_i$ are finite dimensional algebras
over $F$ and $Par(B_i) < Par(A)$ for $i=1,\ldots,r$.
\end{definition}

\begin{remark}
The above definition of a basic algebra, as well as some definitions below, are different from those in \cite{Al-Bel}.

\end{remark}

In Proposition \ref{index-par inequality} we showed that $Ind(A)
\leq Par(A)$ for any finite dimensional algebra. In the next lemma
we show that if $A$ is not basic then the inequality is strict.

\begin{lemma}
Let $A$ be a finite dimensional nonbasic algebra. Then $Ind(A) <
Par(A)$.
\end{lemma}

\begin{proof}
If $A$ is nonbasic, there exists an algebra $B=B_1 \times \cdots
\times B_r$, PI equivalent to $A$, where $B_i$ is a finite
dimensional algebra over $F$ and $Par(B_i) < Par(A)$ for
$i=1,\ldots,r$. We know by Lemma \ref{polynomials with many
alternating sets are identities} that $Ind(B_i) \leq Par(B_i)$.
Suppose $Ind(A)=Par(A)=(d(A), n_{A}-1)$. Then for any $\mu$ (by
definition of the Kemer index), there exists a multilinear
polynomial $f$, nonidentity of $A$, with $\mu$-folds of
alternating sets of cardinality $d(A)$ and precisely $n_{A}-1$
alternating sets of cardinality of $d(A)+1$ (a Kemer polynomial).
We claim however, that if $\mu$ is sufficiently large the
polynomial $f$ is an identity of $B_{i}$, $i=1,\ldots,r$, and
hence an identity of $B=B_1 \times \cdots \times B_r$ showing that
the algebras $A$ and $B$ are not PI equivalent. Indeed, since
$Par(B_i) < Par(A)$ we either have $d(B_i) < d(A)$ or else $d(B_i)
= d(A)$ and $n_{B_{i}} < n_{A}$. In the first case we have
$d(B_i)+1 \leq d(A)$ and hence by Lemma \ref{polynomials with many
alternating sets are identities} $f$ is an identity as long as
$\mu$, the number of alternating sets of cardinality $d(A)$, is
$\geq n_{B_{i}}$. If $d(B_i) = d(A)$ and $n_{B_{i}} < n_{A}$, then
again by Lemma \ref{polynomials with many alternating sets are
identities}, $f$ is an identity of $B_{i}$ since it has $n_{A}-1
\geq n_{B_{i}}$ alternating sets of cardinality
$d(B_{i})+1=d(A)+1$. We see that any Kemer polynomial of $A$ with
$\mu$ alternating sets of cardinality $d(A)$, $\mu \geq \max
\{n_{B_{1}},\dots,n_{B_{r}}\}$, is an identity of $B$. This completes the proof of the lemma.

\end{proof}

Our main task in this section and the next two is to show that if
$A$ is basic then $Ind(A)=Par(A)$. For the proof of that statement
we introduce two properties (of finite dimensional algebras),
named \textit{full} and \textit{property $K$}. We'll show that any
basic algebra $A$ must satisfy both conditions. Then the main task
will be to show that an algebra $A$ which is full and satisfies
property $K$ has $Ind(A)=Par(A)$. Finally, applying the previous
lemma we'll obtain the following equivalences.

\begin{proposition} \label{equivalence conditions-basic}
The following conditions are equivalent for a finite dimensional algebra $A$.

\begin{enumerate}

\item
$A$ is basic

\item

$A$ is full and satisfies property $K$

\item

$Ind(A)=Par(A)$

\end{enumerate}

\end{proposition}

\begin{definition}

We say that a finite dimensional algebra $A$ is \textit{full} if there exists a nonidentity multilinear polynomial $f$ such that every simple
component is represented (among the semisimple substitutions) on \textit{every} nonvanishing evaluation of $f$ on $A$. A
finite dimensional algebra $A$ is said to be full if it is full
with respect to some multilinear polynomial $f$.
\end{definition}

We wish to show that any finite dimensional algebra may be decomposed (up to
\textit{PI}-equivalence) into the direct product of full algebras.

\begin{lemma}\label{decomposition into product of full}

Let $A$ be a finite dimensional algebra over $F$ with $q$ simple components. If the algebra
$A$ is not full, then $A$ is \textit{PI}-equivalent to a finite dimensional algebra $B= B_1 \times \cdots \times B_q$, where
$($1$)$ $B_{i}$ has fewer than $A$
simple components for $i=1,\ldots,q$.
$($2$)$ $d(B_{i})< d(A)$ for $i=1,\ldots,q$ and hence $Par(B_{i})< Par(A)$.

\end{lemma}

\begin{proof}

Note that if $q=0$, i.e. $A$ is radical (nonzero) then $A$ is
full with respect to any multilinear polynomial,
nonidentity of $A$. We therefore may assume that $q > 0$ and
suppose $A$ is not full. This means that any multilinear
polynomial, nonidentity of $A$, has a nonvanishing evaluation which
misses the simple component $A_{i}$ of $A$ for some $i$.

For any $i=1,\ldots, q$, consider the subalgebra $B_{i}=\langle
A_{j},J: j\neq i\rangle$ (i.e. the subalgebra generated by all
elements of $J$ and $A_{j}$, $j\neq i$). Note that $J(B_i)=J(A)=J$
and $B_i/J\cong \overline{B_i}=\prod_{j\neq i}A_{j}$. Consider the
product $B=B_{1} \times \cdots \times B_{q}$. We know that $\Id(A) \subseteq
\Id(B_{i})$, any $i$, and hence $\Id(A) \subseteq
\cap_{i}\Id(B_i)=\Id(B)$. For the converse, let $f$ be a
multilinear polynomial, nonidentity of $A$. By our assumption
above $f$ has a nonzero evaluation on $A$ which misses $A_{i}$ for
some $i$ and hence $f$ is a nonidentity of $B_{i}$. We obtain that
$f$ is a nonidentity of $B$ as desired.

\end{proof}

\begin{corollary}\label{product of Full algebras}
Every finite dimensional algebra $A$ is PI equivalent to a finite dimensional algebra $B=B_1\times \cdots \times B_m$, where $B_{i}$ is full for $i=1,\ldots,m$.

\end{corollary}

\begin{proof}

This follows from the corollary above (either from part (1) or (2)).
\end{proof}

\begin{corollary}

Let $A$ be a finite dimensional algebra. If $A$ is basic then $A$ is full.

\end{corollary}

\begin{proof}
If $A$ is not full, by the previous corollary, $A$ is PI
equivalent to $B=B_1\times \cdots \times B_m$, where
$d(B_{i})<d(A)$ and hence, by definition, $A$ is not basic.

\end{proof}

Roughly speaking, a finite dimensional algebra $A$ is
\textit{full} if its simple components are ``connected'' via all
nonzero evaluations of a suitable nonidentity. ``Property K'' (see
Section \ref{Section: Kemer's Lemma $2$}), concerns with the
number of radical evaluations of nonidentity polynomials.

\end{section}

\begin{section}{Kemer's Lemma $1$}\label{Section: Kemer's Lemma $1$}

The task in this section is to show that if $A$ a finite
dimensional algebra which is full then the first parameter of
$Ind(A)$ and the first parameter of $Par(A)$ coincide.

\begin{proposition}\label{unique point}

Let $A$ be a finite dimensional algebra which is full. Let
$Ind(A)=(\alpha, s)$ and $Par(A)=(d(A), n_{A}-1)$. Then
$\alpha=d(A)$.

\end{proposition}

\begin{proof}

For the proof we need to show that for an arbitrary large integer
$\nu$ there exists a multilinear nonidentity $f$ that contains
$\nu$ folds of alternating sets of cardinality
$\dim_{F}(\overline{A})$.

Since the algebra $A$ is full, there is a multilinear polynomial
$f(x_{1},\ldots,x_{q}, \overrightarrow{y})$, which does not vanish
upon an evaluation of the form $x_{j}= \overline{x}_{j} \in
\overline{A}_{j}$, $j=1,\ldots,q$ and the variables of
$\overrightarrow{y}$ get values in $A$ (either in simple
components or in the radical). The idea is to produce polynomials
$\widehat{f}$'s in the $T$-ideal generated by $f$ which remain
nonidentities of $A$ and that reach eventually the desired form.
The way one checks that the polynomials $\widehat{f}$'s are
nonidentities is by presenting suitable evaluations on which they
do not vanish. Let us reformulate what we need in the following
lemma.

\begin{lemma}[Kemer's Lemma $1$ for finite dimensional algebras] \label{full-folds}

Notation as above. Let $A$ be a finite dimensional algebra which is full with respect to the polynomial $f=f(x_{1},\ldots,x_{q},
\overrightarrow{y})$. Then for any integer $\nu$ there
exists a multilinear polynomial $f^{'}$ in the $T$-ideal generated by $f$ with the following properties:
\begin{enumerate}
\item $f^{'} \notin \Id(A)$
\item $f^{'}$ has $\nu$-folds of alternating sets of cardinality
$\dim_{F}(\overline{A})$.

\end{enumerate}

\end{lemma}

\begin{proof}
We note that if the algebra is radical, then the lemma is clear.

Let $f_{0}$ be the polynomial obtained from $f$ by multiplying (on
the left say) each one of the variables $x_{1},\ldots,x_{q}$ by
variables $z_{1},\ldots,z_{q}$ respectively. Note that the
polynomial obtained, denoted by $f_{1}$, is a nonidentity since
the variables $z_{i}$'s may be evaluated by the elements
${1_{\overline{A}_i}}$'s where

$${1_{\overline{A}_i}} = E^{i}_{1,1}+\cdots + E^{i}_{k_i,k_i}.$$
(Here we use the notation, $\overline{A}_i= M_{k_i}(F)$.)

By linearity there exists a nonzero evaluation where the variables
$z_{1},\ldots,z_{q}$ take values of the form
$E^{1}_{j_1,j_1},\ldots,E^{q}_{j_q,j_q}$ where $1\leq j_i \leq
k_i$ and $i=1,\ldots,q$.

Our aim is to replace each one of the variables
$z_{1},\ldots,z_{q}$ by polynomials $Z_{1},\ldots,Z_{q}$ such
that:

\begin{enumerate}

\item

For every $i=1,\ldots,q$, the polynomial $Z_{i}$ is alternating in
$\nu$-folds of sets of cardinality $\dim_{F}(\overline{A}_{i})$.

\item

For every $i=1,\ldots,q$, the polynomial $Z_{i}$ assumes the value
$E^{i}_{j_i,j_i}$.

\end{enumerate}

Once this is accomplished, we complete the construction by
alternating the corresponding $p$th sets, $p=1,\ldots,\nu$, which come from different $Z_{i}$'s. Clearly,
the polynomial $f^{'}$ obtained

\begin{enumerate}

\item

is a nonidentity since any nontrivial alternation of the evaluated
variables (as described above) vanishes.

\item

$f^{'}$ has $\nu$-folds of alternating sets of cardinality
$\dim_{F}(\overline{A})$.

\end{enumerate}

We now show how to construct the polynomials $Z_{i}$.

In order to simplify the notation we put
$\widehat{A}=\overline{A}_{i}$ ($\cong M_{k}(F)$) where
$\overline{A}_{i}$ is the $i$-th simple component.

Fix $1\leq t \leq k $ and consider a product of the $k^{2}$
different matrix units $E_{i,j}$ of $M_{k}(F)$ with value
$E_{t,t}$ (it is not difficult to show the such a product exists).
We refer to these matrix units $E_{i,j}$ as \textit{designated
matrices}. Next, we border each matrix unit $E_{i,j}$ with
idempotents $E_{i,i}$ and $E_{j,j}$. We refer to these idempotent
matrices as \textit{frames}. Clearly the product of all matrices,
namely designated matrices and frames is $E_{t,t}$. Now we
construct a Capelli polynomial which corresponds to that word of
matrix units: We construct a monomial with variables which
are in $1-1$ correspondence with the word of matrix units just mentioned, namely the designated and frame matrices. We make a
distinction between variables that correspond to designated
matrices ($u_{1}, \ldots, u_{n}, \ldots$) and variables that
correspond to frame matrices ($y_{1}, \ldots, y_{m}, \ldots$). We
denote that monomial by $\Sigma_{1}{t}$. The subscript $1$ stands
for the fact that the monomial has exactly one set of designated
variables whereas the subscript $t$ stands for the fact that there
is an evaluation with value $E_{t,t}$. Next we consider the product of
$\nu$ monomials $\Sigma_{1}{t}$ (with different variables). We denote the long monomial
obtained by $\Sigma_{\nu}{t}$. In order to complete the construction of the polynomials $Z_{i}$, we consider the polynomial $\hat{Z}_{\nu}{t}$ obtained by alternating each set of
designated variables \textit{separately}. We
let $Z_{i}=\hat{Z}_{\nu}{j_{i}}$, for each $i=1,\ldots,q$. Once again for each
$p=1,\dots,\nu$, we consider the $p$th set of designated variables
in each polynomial $Z_{i}$ and as indicated above we alternate
these variables among the different $i=1, \ldots, q$. It is clear
that if we evaluate the variables accordingly, any nontrivial
permutation yields a zero value and hence get a nonidentity of $A$
of the desired form.

\end{proof} \end{proof}

\begin{remark}
Let us return once again to the definition of the \textit{full}
property of a finite dimensional algebra $A$. By definition, a
finite dimensional algebra $A$ is full if there exists a
multilinear polynomial, nonidentity of $A$ such that all simple
components are represented in any nonzero evaluation on $A$.
Nevertheless, for the proof of Kemer's Lemma $1$ we used a
seemingly weaker condition, namely the existence of a multilinear
nonidentity of $A$ which has a nonzero evaluation which visits all
simple components of $A$. These two condition are indeed
nonequivalent for a given polynomial. However, as we see below, it follows from
Kemer's Lemma 1, the $T$-ideal generated by a
nonidentity $f$ which satisfies the weaker condition contains a
polynomial $f'$ which satisfies the stronger condition.

\end{remark}

\begin{corollary}

Let $A$ be a finite dimensional algebra. Then $A$ is full if and
only if there exists a multilinear polynomial $f$, nonidentity of $A$ which \textit{admits} a nonzero evaluation which visits every simple components of $A$.

\end{corollary}

\begin{proof}
Clearly, if $A$ is full the condition is satisfied. In order to prove the
opposite direction we need to show that if $f$ is
a multilinear polynomial, nonidentity of $A$ with a nonzero
evaluation on $A$ which visits any simple component then there
exists a multilinear polynomial $f'$, nonidentity of $A$, which
visits every simple component of $A$ on every nonzero evaluation.
Indeed, by Kemer lemma $1$ we know there exists a multilinear
polynomial $f'$ in the $T$-ideal $\langle f \rangle$, nonidentity
of $A$, which alternates on $\nu$ disjoint sets of variables of
cardinality $d(A)$, where $\nu$ is arbitrary. Now, if we take $\nu
\geq n_{A}$, then in any nonzero evaluation of $f'$, at least one
alternating set of cardinality $d(A)$ must be evaluated by only
semisimple elements for otherwise in each each alternating set
there is a variable which gets a radical value. Of course any such
evaluation vanishes since at least $n_{A}$ variables of $f'$ get
radical values. But if in any nonzero evaluation of $f'$, there is
one alternating set of cardinality $d(A)$ which gets only
semisimple values, by the alternation, these values must be
linearly independent over $F$ and hence consist of a basis of the
semisimple subalgebra $\overline{A}$. This shows that any nonzero
evaluation of $f'$ visits all simple components as desired.

\end{proof}
\end{section}

\begin{section}{Kemer's Lemma $2$}\label{Section: Kemer's Lemma $2$}

In this section we prove Kemer's Lemma $2$. Before stating the
precise statement we need an additional reduction which enables us
to control the number of radical evaluations in certain
nonidentities.

Let $f$ be a multilinear polynomial which is not in $\Id(A)$.
Clearly, any nonzero evaluation cannot have more than $n_{A}-1$
radical evaluations.

\begin{lemma} Let $A$ be an algebra which is full. Let $Ind(A)=(\alpha, s)$ be its Kemer index.
Then $s \leq n_{A}-1$.

\end{lemma}

\begin{proof}

By Kemer's Lemma $1$ we can find nonidentity
polynomials with arbitrary many alternating sets of cardinality
$d(A)$ and since this is the maximum possible, we have that
$\alpha=d(A)$. It follows that in alternating sets of cardinality
$d(A)+1$ we must have at least one radical evaluation and hence we
cannot have more than $n_{A}-1$ in a nonidentity.

\begin{remark}
Although not needed later in the paper, it worth noting that the result of the lemma above holds also for arbitrary
finite dimensional algebras. Indeed we know (Corollary
\ref{product of Full algebras}) that the algebra $A$ is PI equivalent to the
direct product of algebras $B_{1} \times \cdots \times B_{m}$,
where $B_{i}$ is full for $i=1,\ldots,m$. For each $B_{i}$ we
consider the dimension of the semisimple part $d(B_{i})$. Applying
Kemer lemma 1 we have that $\alpha \geq max_{i}(d(B_{i}))$. On the
other hand if $\alpha > d(B_{i})$, any multilinear polynomial with
more than $n_{B_{i}}-1$ alternating sets of cardinality $\alpha$
is in $\Id(B_{i})$ (any alternating set must have at least one
radical evaluation) and hence if $ \alpha > max_{i}(d(B_{i}))$,
any polynomial as above is an identity of $B_{1} \times \cdots
\times B_{m}$ and hence of $A$. This contradicts the definition of
the parameter $\alpha$ and hence $\alpha = max_{i}(d(B_{i}))$. Now
take an alternating set of cardinality $\alpha +1$. In every such
set we must have a radical evaluation or elements from different
full algebras. If they come from different full algebras we get
zero. If we get a radical element then we cannot pass $n_{A}-1$.
\end{remark}
\end{proof}

The next definition is key in the proof of Kemer's Lemma $2$ (see below).
\begin{definition}

Notation as above. Let $f$ be a multilinear polynomial which is not in
$\Id(A)$. We say that $A$ has property $K$ with respect to $f$ if $f$ vanishes on
any evaluation on $A$ with less than $n_{A}-1$ radical substitutions.

We say that a finite dimensional algebra $A$ has property
$K$ if it satisfies the property with respect to some nonidentity
multilinear polynomial.

\end{definition}

\begin{proposition}\label{property K}

Let $A$ be a finite dimensional basic
$F$-algebra. Then it has property $K$.

\end{proposition}

Before proving the proposition we introduce a construction which
will enable us to put some ``control'' on the nilpotency index of
(the radical of) finite dimensional algebras which are PI
equivalent.

Let $B$ be any finite dimensional algebra and let
$B^{'}=\overline{B}*F\{x_{1},\dots,x_{n}\}$ be the algebra of
polynomials in the variables $\{x_{1},\dots,x_{n}\}$ with coefficients in
$\overline{B}$, the semisimple component of $B$ (in case
$\overline{B}=0$, $B^{'}=F\langle x_{1},\dots,x_{n}\rangle$ is the
nonunital free $F$-algebra generated by the variables
$\{x_{1},\dots,x_{n}\}$). The number of variables we take is at
least the dimension of $J(B)$. Let $I_{1}$ be the ideal of $B^{'}$
generated by all evaluations of polynomials of $\Id(B)$ on $B^{'}$
and let $I_{2}$ be the ideal generated by all variables
$\{x_{1},\ldots,x_{n}\}$. Consider the algebra $\widehat{B}_{u}=B^{'}/(I_{1} +
I_{2}^{u})$.

\begin{proposition}\label{Control on nilpotency index}
The following hold.
\begin{enumerate}

\item

$\Id(\widehat{B}_{u})= \Id(B)$ whenever $u \geq n_{B}$ $($$n_{B}$
denotes the nilpotency index of $J(B)$$)$. In particular
$\widehat{B}_{u}$ and $B$ have the same index.

\item

$\widehat{B}_{u}$ is finite dimensional.

\item

The nilpotency index of $J(\widehat{B}_{u})$ is $\leq u$.

\end{enumerate}

\end{proposition}

\begin{proof}
Note that by the definition of $\widehat{B}_{u}$ (modding $B^{'}$
by an ideal that contains $I_{1}$), $\Id(\widehat{B}_{u}) \supseteq
\Id(B)$. On the other hand there is a surjection $\widehat{B}_{u}
\longrightarrow B$ which maps the variables $\{x_{1},\ldots,x_{n}\}$ onto a
spanning set of $J(B)$ and $\overline{B}$ is mapped
isomorphically. Indeed, we have such a map from $B^{'}$, namely
there is a surjective map $\phi: B^{'}\longrightarrow B$ which
maps $\overline{B}$ isomorphically and the variables $\{x_{1},\ldots,x_{n}\}$
onto a spanning set of $J(B)$. The ideal $I_{1}$ consists of all
evaluation of $\Id(B)$ on $B^{'}$ and hence is contained in
$ker(\phi)$. Also the ideal $I_{2}^{u}$ is contained in
$ker(\phi)$ since $u \geq n_{B}$ and $\phi(x) \in J$. This shows
(1).

To see
(2) observe that any element in $\widehat{B}_{u}$ is represented
by a sum of elements of the form $b_{1}z_1b_2z_2 \cdots
b_{j}z_jb_{j+1}$ where $j < u$, $b_{i}\in \overline{B}$ and $z_i
\in \{x_{i}\}$ (we allow also consecutive $z_{i}$'s). Clearly, the
subspace spanned by monomials for a given configuration of the
$z_i$'s (and arbitrary $b_i$'s) has finite dimension. On the other
hand the number of different configurations is finite and so the
result follows. In order to prove the 3rd statement, note that
$I_{2}$ generates a radical ideal in $\widehat{B}_{u}$  and since
$B^{'}/I_{2}\cong \overline{B}$ we have that
$\widehat{B}_{u}/I_{2} \cong
B^{'}/(I_{1}+I_{2}^{u}+I_{2})=B^{'}/(I_{1}+I_{2})\cong
(B^{'}/(I_{2}))/I_{1}= \overline{B}/I_{1}=\overline{B}$ (the last
equality follows from the fact that $\overline{B} \subseteq B$).
We therefore see that $I_{2}$ generates the radical in
$\widehat{B}_{u}$, and hence its nilpotency index is bounded by
$u$ as claimed.
\end{proof}

\begin{proof}(of Proposition \ref{property K})
Suppose $A$ is a basic algebra for which $K$ fails.
 This means that any multilinear
polynomial that vanishes on any evaluation with less than
$n_{A}-1$ radical evaluations must be in $\Id(A)$. Consider the
algebra $\widehat{A}_{u}=A^{'}/(I_{1}+I_{2}^{u})$ (from the
proposition above). We claim that
$\Id(\widehat{A}_{n_{A}-1})=\Id(A)$. This will show that $Par(\widehat{A}_{n_{A}-1}) < Par(A)$ (and hence $A$ is not basic) since the algebras $\widehat{A}_{n_{A}-1}$ and
$A$ have isomorphic semisimple parts and the nilpotency index of
$J(\widehat{A}_{n_{A}-1})$ is bounded by $n_{A}-1$.

To prove the claim, note that by construction $\Id(A) \subseteq
\Id(\widehat{A}_{n_{A}-1})$. For the converse take a multilinear
polynomial $f$ which is not in $\Id(A)$. Then by assumption, there
is a nonzero evaluation $\tilde{f}$ of $f$ on $A$ with
\textit{less} than $n_{A}-1$ radical substitutions (say $k$).
Following this evaluation we refer to the variables of $f$ that
get semisimple(radical) values as semisimple (radical) variables
respectively. Consider the evaluation $\hat{f}$ of $f$ on
$A^{'}=\overline{A}*\{x_{1},\dots,x_{n}\}$ where semisimple
variables are evaluated as in $\tilde{f}$ whereas the radical
variables are evaluated on $\{x_{1},\dots,x_{n}\}$ respecting the
surjection $\phi: A^{'}\rightarrow A$ (by abuse of language ``we
replace the radical values in $\tilde{f}$ by indeterminates'').

We claim $\hat{f} \notin I_{1}+I_{2}^{n_{A}-1}$. This will show $f
\notin \Id(\widehat{A}_{n_{A}-1})$ which is what we want. Clearly,
the map $\phi$ induces $\bar{\phi}: A^{'}/I_{1} \rightarrow A$ and
hence $\hat{f} \notin I_{1}$. Next, note (by multiplication with
central indeterminates) that an element in
$A^{'}=\overline{A}*\{x_{1},\dots,x_{n}\}$ is in $I_{1}$ if and
only if each one of its multihomogeneous components in the
variables of $\{x_{1},\dots,x_{n}\}$ is in $I_{1}$. But by
construction $\hat{f}$ is multihomogeneous of degree $k < n_{A}-1$
in the variables $\{x_{1},\dots,x_{n}\}$ whereas any element of
$I_{2}^{n_{A}-1} \subseteq A^{'}$ is the sum of multihomogeneous
elements degree $\geq n_{A}-1$. We therefore have that $\hat{f}
\in I_{1}+I_{2}^{n_{A}-1}$ if and only if $\hat{f} \in I_{1}$ and
we are done.

\end{proof}

Let $A$ be a basic algebra. Let $\Par_{A}=(d(A), n_{A}-1)$ where $d(A)$ is the
dimension of the semisimple part of $A$ and $n_{A}$ the nilpotency
index of $J(A)$. By Proposition \ref{property K} the algebra
satisfies property $K$ with respect to a nonidentity polynomial
$f$, that is $f$ vanishes on any evaluation whenever it has less
than $n_{A}-1$ radical substitutions. Furthermore, there is
\textit{possibly a different} nonidentity polynomial $h$ with respect to
which $A$ is full, that is $h$ has a nonzero evaluation which
``visits'' each one of the simple components of
$\overline{A}$. In order to proceed we need both properties to be
satisfied by the same polynomial.

\begin{lemma}\label{Full and property K}
Let $A$ be a basic algebra. Then there exists a multilinear
polynomial $f$, nonidentity of $A$, which visits every simple
component in any nonzero evaluation and has property $K$

\end{lemma}

\begin{proof}

We prove the lemma by showing that there exists a multilinear
polynomial, nonidentity of $A$, which visits every simple
component in any nonzero evaluation and in addition realizes the
property $K$ of $A$, that is, vanishes on any evaluation with less
than $n_{A}-1$ radical evaluations.

Suppose first $A$ is a radical algebra. Then, any nonidentity of
$A$ which satisfies property $K$, visits (in an empty way) all
simple components of $A$ in every nonzero evaluation and so we are
done in that case. Suppose now $q>0$ (number of simple components
of $A$) and suppose by way of contradiction that the lemma is
false, that is any multilinear polynomial $f$, nonidentity of $A$,
has a nonzero evaluation with less than $n_{A}-1$ radical
evaluations or has a nonzero evaluation which does not visit all
simple components of $A$.

Consider the subalgebras $B_{i}$, $i=1,\ldots,q$ constructed in
the proof of Lemma \ref{decomposition into product of full}, the algebra $B=B_1\times \cdots \times B_q$
and the algebra $\widehat{A}_{n_{A}-1}$ (see proof of Prop. \ref{property K}). Let
$C=B_1\times \cdots B_q \times \widehat{A}_{n_{A}-1}$. We claim
the algebras $A$ and $C$ are PI equivalent. This will contradict
the fact that $A$ is basic since $Par(\widehat{A}_{n_{A}-1})<
Par(A)$ and $Par(B_i) < Par(A)$ for $i=1,\ldots,q$. To prove the
claim note that by the definition of the $B_i$'s and
$\widehat{A}_{n_{A}-1}$ we have $\Id(A) \subseteq
(\cap\Id(B_i))\cap \Id(\widehat{A}_{n_{A}-1})= \Id(C)$. For the
opposite direction let $f \notin \Id(A)$, multilinear. By
assumption, $f$ has either a nonvanishing evaluation on $A$ which
does not visit all simple components of $A$ (say $A_j$) or has a
nonvanishing evaluation with less than $n_{A}-1$ radical
evaluations. In the first case $f$ is a nonidentity of $B_{j}$
whereas in the other case $f$ is a nonidentity of
$\widehat{A}_{n_{A}-1}$. In both cases $f \notin
\Id(C)$ and the lemma is proved.
\end{proof}

\begin{example}
Let $A$ be the algebra over $F$ of upper triangular $2 \times 2$-matrices. Consider the polynomial $p(x,y,z)=xyz$. It is clear that $p$ has a nonzero evaluation which visits the two simple components of $A$ ($x=e_{11}, y=e_{12}, z=e_{22}$). On the other one can easily find a nonzero evaluation of $p$ which visits only one simple component ($x=e_{11}, y=e_{11}, z=e_{11}$). Next we construct a polynomial which visits both simple components on every nonzero evaluation. Let $q(x_1,x_2,x_3)=\sum (-1)^{\sigma}x_{\sigma(1)}x_{\sigma(2)}x_{\sigma(3)}$. The polynomial $q$ is alternating on a set of cardinality $3$, the variables $x_1, x_2, x_3$ must get linearly independent elements of $A$ in any nonzero evaluation. Since $A$ is of dimension $3$, the elements $e_{11},e_{12},e_{22}$ must appear as values and hence any nonzero evaluation visits all simple components. It remains to show that there is at least one nonzero evaluation. Indeed, it is easily verified that the evaluation $x_1=e_{11}, x_2=e_{12}, x_3=e_{22}$ is nonzero.

\end{example}

We can now state and prove Kemer's Lemma $2$.

\begin{lemma} [Kemer's Lemma $2$] \label{Kemer's Lemma 2}

Let $A$ be a finite dimensional. Suppose $A$ is basic whose index
is $(d=d(A), n_{A}-1)$. Then for any integer $\nu$ there exists a
multilinear, nonidentity polynomial $f$ with $\nu$-alternating
sets of cardinality $d$ $($small sets$)$ and precisely $n_{A}-1$
alternating sets of variables of cardinality $d+1$ $($big sets$)$.

\end{lemma}
The theorem is clear either in case $A$ is radical or semisimple (i.e. simple). Hence for the proof
we assume that $q \geq 1$ (the number of simple components of $A$) and $n_{A} >1$.
\begin{note}

Any nonzero evaluation of such $f$ must consists only of
semisimple evaluations in the $\nu$-folds and each one of the big
sets (namely the sets of cardinality $d+1$)  must have exactly
one radical evaluation.

\end{note}

\begin{proof}(of Lemma \ref{Kemer's Lemma 2})

By Lemma \ref{Full and property K}, there exists a multilinear polynomial $f$ with respect to which $A$ is full and has property
$K$. Let us fix a nonzero evaluation $x \longmapsto \widehat{x}$
realizing the ``full'' property. Note that by the construction of
$f$, being the evaluation nonzero, precisely $n_{A}-1$ variables
must obtain radical values, and hence the rest of the variables obtain
semisimple values. Let us denote by $w_{1},\ldots, w_{n_{A}-1}$ the
variables that obtain radical values (in the evaluation above) and
by $\widehat{w}_{1},\ldots, \widehat{w}_{n_{A}-1}$ their
corresponding values. By abuse of language we refer to the
variables $w_{1},\ldots, w_{n_{A}-1}$ as \textit{radical}
variables.

\begin{remark}
Note that by Kemer lemma $1$ we could assume at this point that
$f$ is alternating on $\nu$-folds of alternating sets of
cardinality $\dim_{F}(\overline{A})$, but since it will be
important where these alternating sets are located (with respect
to the radical evaluations), our starting polynomial in the proof
below is merely assumed to realize property $K$ and the nonzero
evaluation (fixed above) to realize the full property of $A$.

\end{remark}

We will consider four cases.
These correspond to whether $A$ has or does not have an identity
element and whether $q$ (the number of simple components) $>1$ or
$q=1$.

{\bf Case $(1,1)$} ($A$ has an identity element and $q > 1$).

By linearity we may assume the evaluation of any radical variable
$w_{i}$ is of the form $1_{A_{j(i)}}\widehat{w}_{i}1_{A_{\widetilde{j}(i)}}
$, $i=1,\ldots,n_{A}-1$, where $1_{A_{k}}$ is the identity element of the simple component
$A_{k}$. Note
that the evaluation remains full (i.e. visits every simple component
of $A$).

Choose a monomial $X$ of $f$ which does not vanish upon the above
evaluation.

Notice that the variables of $X$ which get
semisimple evaluations from \textit{different} simple components must be
separated by radical variables.

\begin{claim}

The elements $1_{A_{j(i)}}, 1_{A_{\widetilde{j}(i)}}$, $i=1,\ldots,n_{A}-1$, which appear in the borderings
above, represent all simple components of $\overline{A}$.

\end{claim}

Indeed, suppose that the component $A_{1}$ (say) is not
represented among the $1_{A_{k}}$'s. Since our original
evaluation is full, there is a variable which is evaluated by an
element $u$ of $A_{1}$. ``Moving'' along the monomial $X$ to the
left or right of $u$ we will hit a bordering value $1_{A_{k}}$
before we hit any radical evaluation. But this
is possible only if both $u$ and $1_{A_{k}}$ belong to the
same simple component. This proves the claim.

But we need more: Consider the radical evaluations which are
bordered by pairs of elements $(1_{A_{j(i)}}, 1_{A_{\widetilde{j}(i)}})$ where $j(i)\neq \widetilde{j}(i)$ (i.e. belong to different simple components).

\begin{claim}

Every simple component is represented by one of the elements in
these pairs.

\end{claim}

Again, assume that $A_1$ is not represented among these pairs. By
the preceding claim $A_1$ is represented in some pair and so it must be
represented by both partners in each pair it appears. Take such a
pair $(1_{A_{j(i)}}, 1_{A_{\widetilde{j}(i)}})$, where $j(i)=\widetilde{j}(i)=1$. Moving along the
monomial $X$ to the left of $1_{A_{j(i)}}$ or to the right $1_{A_{\widetilde{j}(i)}}$
we will hit a value in a different
simple component. But before that we must hit a radical evaluation
which is bordered by a pair where one of the partners is from
$A_1$ and the other from a different simple component. This
contradicts our assumption and hence the claim is proved.

For $t=1,\ldots,q$ we fix a variable $w_{r_t}$ whose radical
value is $1_{A_{j(r_t)}}\widehat{w}_{r_t}1_{A_{\widetilde{j}(r_t)}}$

where

\begin{enumerate}

\item

$j(r_t)\neq \widetilde{j}(r_t)$ (i.e. different simple components).

\item

One of the element $1_{A_{j(r_t)}},1_{A_{\widetilde{j}(r_t)}}$
is the identity element of the $t$-th simple component. We
refer to that element as the \textit {idempotent
attached} to the simple component $A_{t}$.

\end{enumerate}

\begin{remark}

Note that we may have $w_{r_t}=w_{r_{t^{'}}}$ even if $t\neq t^{'}$.

\end{remark}

Next by the $T$ operation we replace the variables $w_{r_t}$,
$t=1,\ldots,q$, by $y_{r_t}w_{r_t}$ or $w_{r_t}\widetilde{y}_{r_t}$
according to the location of the primitive idempotent attached to
the $t$-th simple component. Clearly, by evaluating the variable
$y_{r_t}$ by $1_{A_{j(r_t)}}$ (or the variable
$\widetilde{y}_{r_t}$ by $1_{A_{\widetilde{j}(r_t)}}$) the
value of the polynomial remains the same and in particular
nonzero.

Applying Lemma \ref{full-folds} we can replace the variable
$y_{r_t}$, $t=1,\dots,q$, by a $\nu$-fold alternating polynomial (on the sets
$U^{t}_{l}$) $Z_{r_t}=Z_{r_t}(U^{t}_1,\ldots, U^{t}_{\nu}; Y_t)$,
(or replace $\widetilde{y}_{r_t}$ by a $\nu$-fold alternating
polynomial $\widetilde{Z}_{r_t}$). Here, the sets $U^{t}_{l}$,
$l=1,\ldots, \nu$ are each of cardinality $\dim_{F}(A_{t})$. Now,
if we further alternate the sets $U^{1}_{l},\ldots,U^{q}_{l}$ for
$l=1,\ldots, \nu$, we obtain a nonidentity polynomial with
$\nu$-folds of (small) sets of alternating variables where each
set is of cardinality $\dim_{F}(\overline{A})$. In the sequel we fix
an evaluation of the polynomials $Z_{r_t}$ (or
$\widetilde{Z}_{r_t}$) so the entire polynomial obtains a nonzero
value. As in Kemer's Lemma $1$, also here the alternating
variables $U^{t}_{l}$ are bordered by variables (called
\textit{frames}) whose values are primitive idempotents in
$A_{t}$.

Our next task is to construct such polynomial with an extra $n_{A}-1$ alternating sets of cardinality $d+1$ (big sets).
Consider the radical variables $w_{r_t}$, $t=1,\dots,q$ with radical
evaluations $1_{A_{j(r_t)}}\widehat{w}_{r_t}1_{A_{\widetilde{j}(r_t)}}$, $j(r_t) \neq \widetilde{j}(r_t)$ (i.e. different simple components).

We attach each variable $w_{r_t}$ to one alternating set
$U^{1}_{l},\ldots,U^{q}_{l}$ (some $l$). We see that any
nontrivial permutation of $w_{r_t}$ with one of the variables of
$U^{1}_{l},\ldots,U^{q}_{l}$, keeping the evaluation above, will
yield a \textit{zero value} since the primitive idempotents values
in frames variables of each variable of
$U^{1}_{l},\ldots,U^{q}_{l}$ belong to the same simple components
whereas the pair of idempotents in
$1_{A_{j(r_t)}}\widehat{w}_{r_t}1_{A_{\widetilde{j}(r_t)}}$ belong
to different simple components. Thus we may alternate the variable
$w_{r_t}$ with $U^{1}_{l_t},\ldots,U^{q}_{l_t}$, $t=1,\ldots, q$
and obtain a multilinear nonidentity of $A$. Next we proceed in a
similar way with any variable $w_{i}$ whose evaluation is
$1_{A_{j(i)}}\widehat{w}_{i}1_{A_{\widetilde{j}(i)}}$ and
$j(i)\neq\widetilde{j}(i)$.

Finally we need to attach the radical variables $w_{i}$ whose
evaluation is
$1_{A_{j(i)}}\widehat{w}_{i}1_{A_{\widetilde{j}(i)}}$ where
$j(i)=\widetilde{j}(i)$ (i.e. the same simple component) to some
small set. We claim also here that if we attach the variable
$w_{i}$ to the sets $U^{1}_{l},\ldots,U^{q}_{l}$ (some $l$), any
nontrivial permutation yields a zero value, and hence the value of
the entire polynomial remains unchanged. If we permute $w_{i}$
with an element $u_{0} \in U^{k}_{l}$ which is bordered by
idempotents different to $1_{A_{j(i)}}$ we obtain zero. On the
other hand we claim that any permutation of $w_{i}$ with an
element $u_{0} \in U^{k}_{l}$ which is bordered by the idempotent
$1_{A_{j(i)}}$ corresponds to an evaluation of the original
polynomial with fewer radical values and then we will be done by
the property $K$. In order to simplify our notation let
$\{U^{1}_{l},\ldots,U^{q}_{l}\}=\{U^{1},\ldots,U^{q}\}$ (omit the
index $l$) and suppose without loss of generality, that $u_{0}\in
U^{1}$. Permuting the variables $w_{i}$ and $u_{0}$ (with their
corresponding evaluations) we see that the polynomial
$Z_{r_1}=Z_{r_1}(U^{1}=U^{1}_1,\ldots, U^{t}_{\nu}; Y_t)$ (or
$\widetilde{Z}_{r_1}$) with $w_{i}$ replacing $u_{0}$, obtains a
radical value which we denote by $\widehat{\widehat{w}}$.
Returning to our original polynomial, we obtain the same value if
we evaluate the variable $w_{i}$ by a suitable semisimple element,
the variable $w_{r_1}$ by $\widehat{\widehat{w}}\widehat{w}_{r_1}$
(or $\widehat{w}_{r_1}\widehat{\widehat{w}}$) and the evaluation
of any semisimple variable remains semisimple. It follows that if
we make such a permutation for a unique radical variable $w_{i}$,
the value amounts to an evaluation of the original polynomial with
$n_{A}-2$ radical evaluations and hence vanishes. Clearly,
composing $p>0$ permutations  of that kind yields a value which
may be obtained by the original polynomial $f$ with $n_{A}-1-p$
radical evaluations and hence vanishes by property $K$. This
completes the proof of the lemma where $A$ has identity and $q$,
the number of simple components, is $> 1$.

\smallskip

{\bf Case $(1,2)$} ($A$ has an identity element and $q=1$). We start with a
nonidentity $f$ which satisfies property $K$. Clearly we may
multiply $f$ by a variable $x$ and get a nonidentity (since
$x$ may be evaluated by $1$). Again by Lemma \ref{full-folds}
we may replace $x$ by a polynomial $h$ with $\nu$-folds of
alternating sets of cardinality $d$. Consider
the polynomial $hf$. We attach the radical variables of $f$ to
some of the small sets in $h$. Any nontrivial permutation
vanishes because $f$ satisfies property $K$. This completes the
proof of the Lemma \ref{Kemer's Lemma 2} in case $A$ has an
identity element.


\smallskip

{\bf Case $(2,1)$}. Suppose now $A$ has no identity element and
$q>1$. The proof in this case is basically the same as in the case
where $A$ has an identity element. Let
$e_{0}=1-1_{A_1}-1_{A_2}-\cdots -1_{A_q}$ ($1\in F$) and attach $e_{0}$ to
the set of elements which border the radical values
$\widehat{w}_{j}$. Of course $e_{0}$ is not an element of $A$ but the product $e_{0}a$, $a \in A$ is well defined. A similar argument shows that also
here every simple component ($A_{1},\ldots,A_{q}$) is
represented in one of the bordering pairs where the partners are
\textit{different} (the point is that one of the partners (among
these pairs) may be $e_{0}$). Now we complete the proof exactly as
in case $(1,1)$.

\smallskip

{\bf Case $(2,2)$}. In order to complete the proof of the lemma we
consider the case where $A$ has no identity element and $q=1$. The
argument in this case is different. For simplicity we
denote by $e_1=1_{A_1}$ and $e_0=1-e_1$. Let
$f(x_{1},\ldots,x_{n})$ be a nonidentity of $A$
which satisfies property $K$ and let
$f(\widehat{x}_{1},\ldots,\widehat{x}_{n})$ be a
nonzero evaluation for which $A$ is full. If
$e_{1}f(\widehat{x}_{1},\ldots,\widehat{x}_{n})
\neq 0$ (or
$f(\widehat{x}_{1},\ldots,\widehat{x}_{n})e_{1}$)
we proceed as in case $(1,2)$. To treat the remaining case we may
assume further that

$$
e_{0}f(\widehat{x}_{1},\ldots,\widehat{x}_{n})e_{0}\neq 0.
$$

First note, by linearity, that each one of the radical values
$\widehat{w}$ may be bordered by one of the pairs
$\{(e_{0},e_{0}), (e_{0},e_{1}), (e_{1},e_{0}), (e_{1},e_{1})\}$
so that if we replace the evaluation $\widehat{w}$ (of $w$) by the
corresponding element $e_{i}\widehat{w}e_{j}$, $i,j=0,1$, we get
nonzero.

Now, if one of the radical values (say $\widehat{w_{0}}$) in
$f(\widehat{x}_{1},\ldots,\widehat{x}_{n}$) allows a bordering by
the pair $(e_{0},e_{1})$ (and remains nonzero), then replacing
$w_{0}$ by $w_{0}y$ yields a nonidentity (since we may evaluate $y$ by $e_{1}$). Invoking Lemma
\ref{full-folds} we may replace the variable $y$ by a polynomial
$h$ with $\nu$-folds of alternating (small) sets of cardinality
$\dim_{F}(\overline{A})=\dim_{F}(A_{1})$. Then we attach the
radical variable $w_{0}$ to a suitable small set. Clearly, the
value of any nontrivial permutation of $w_{0}$ with any element of the small set is zero since the
borderings are different. Similarly, attaching radical variables
$w$ whose radical value is $e_{i}\widehat{w}e_{j}$ where $i\neq
j$, to small sets yields zero for any nontrivial permutation and
hence the value of the polynomial remains nonzero. The remaining
possible values of radical variables are either
$e_{0}\widehat{w}e_{0}$ or $e_{1}\widehat{w}e_{1}$. Note that
since semisimple values can be bordered only by the pair $(e_{1},
e_{1})$, any alternation of the radical variables whose radical
value is $e_{0}\widehat{w}e_{0}$ with elements of a small set
vanishes and again the value of the polynomial remains unchanged.
Finally (in order to complete this case, namely where one of the radical variables, say $w_{0}$, is bordered by the pair $(e_{0},e_{1})$) we attach the remaining radical
variables (whose values are bordered by $(e_{1},e_{1})$) to suitable small sets in $h$. Here,
the value of any nontrivial permutation of $w_{0}$ with elements of the small set is zero because of property $K$ (as in case $(1,1)$). This settles
the case where the bordering pair of $\widehat{w}_{0}$ is $(e_{0},e_{1})$ . Obviously, the same holds if the bordering pair of $\widehat{w}_{0}$ above
is $(e_{1},e_{0})$. The outcome is that we may assume that all
radical values may be bordered by either $(e_{0},e_{0})$ or
$(e_{1},e_{1})$.

\begin{claim}
Under the above assumption, all pairs that border radical values are equal, that is are all $(e_{0},e_{0})$ or all $(e_{1},e_{1})$.
\end{claim}

Indeed, if we have of both kinds, we must have a radical value
which is bordered by a mixed pair since the semisimple variables
can be bordered only by the pair $(e_{1},e_{1})$ (and in particular they cannot be bordered by mixed pairs). This of course contradicts our assumption.

A similar argument shows that we cannot have radical variables $w$
with values $e_{0}\widehat{w}e_{0}$ since again, semisimple values
can be bordered only by $(e_{1}, e_{1})$ and this will force the
existence of a radical value bordered by mixed idempotents (we remind the reader that here $q=1$ and $A$ is full with respect the evaluation $f(\widehat{x}_{1},\ldots,\widehat{x}_{n})$).

The remaining case is the case where all values (radical and
semisimple) are bordered by the pair $(e_{1}, e_{1})$ and this
contradicts the assumption
$e_{0}f(\widehat{x}_{1},\ldots,\widehat{x}_{n})e_{0}\neq 0$. This
completes the proof of the Lemma.

\end{proof}

We can now prove Proposition \ref{equivalence conditions-basic}.

\begin{proof}
If $A$ is basic we know that $A$ is full and satisfies property
$K$. Kemer's Lemma $2$ shows that if $A$ is full and satisfies
property $K$ there exist nonidentities of $A$ with arbitrary many
alternating sets of cardinality $d(A)$ and precisely $n_{A}-1$
alternating sets of cardinality $d(A)+1$. This shows that
$Ind(A)=Par(A)$. In order to complete the proof of the proposition
we show that if a finite dimensional algebra $A$ admits a
multilinear polynomial $f$ with $\mu$ alternating sets of
cardinality $d(A)$ and precisely $n_{A}-1$ alternating sets of
cardinality $d(A)+1$ then $A$ is basic. Suppose not. Then $A$ is
PI equivalent to an algebra $B=B_1 \times \cdots \times B_m$ where
$Par(B_i) < Par(A)$, $i=1,\ldots,m$. But this implies that $f \in
\Id(B_{i})$, $i=1,\ldots,m$ and hence $f \in \Id(B)$.
Contradiction.

\end{proof}

We close this section  by establishing the Phoenix property of
Kemer polynomials. As mentioned above, this seemingly
``unimportant'' property is in fact key for the entire proof of
the representability theorem.

\begin{lemma}\label{Phoenix}

Let $A$ be a finite dimensional algebra over $F$ and suppose $A$ is basic. The following hold.

\begin{enumerate}

\item

Let $f \notin \Id(A)$ be a multilinear polynomial and suppose $A$
is full with respect to \underline{\textit{any}} nonzero evaluation of $f$ on
$A$, that is, in any nonzero evaluation of $f$ on $A$ we must have
semisimple values from all simple components. Then if $f' \in
\langle f \rangle$ is multilinear $($$\langle f \rangle =$ $T$-
ideal generated by $f$$)$ is a nonidentity of $A$ then it is full with respect to
any nonzero evaluation on $A$.

\item

Let $f \notin \Id(A)$ be multilinear and suppose $f$ vanishes on
any evaluation on $A$ with less than $n_{A}-1$ radical
evaluations. Then if $f' \in \langle f\rangle$ is multilinear $($and
nonidentity of $A$$)$ then it vanishes on any evaluation on $A$ with
less than $n_{A}-1$ radical evaluations.

\end{enumerate}

\end{lemma}

\begin{proof}

Suppose $f(x_{1},\ldots,x_{n})$ is a multilinear polynomial which
satisfies the condition in $(1)$. It is sufficient to show the condition
remains valid if $f'$ is multilinear and has the form (a)
$f'=\sum_{i}g_{i}\cdot f\cdot h_{i}$ (b)
$f'(z_{1},\ldots,z_{t},x_{2},\ldots,x_{n}) =
f(Z,x_{2},\ldots,x_{n})$ where $Z=z_{1}\cdots z_{t}$ is a
multilinear monomial consisting of variables disjoint to the
variables of $f(x_{1},\ldots,x_{n})$. If $f'=\sum_{i}g_{i}\cdot
f\cdot h_{i}$ then any nonzero evaluation of $f'$ arises from a
nonzero evaluation of $f$ and so the claim is clear in this case.
Let
$f'(z_{1},\ldots,z_{t},x_{2},\ldots,x_{n})=f(Z,x_{2},\ldots,x_{n})$
and suppose $x_{i}=\hat{x}_{i}$ and $z_{i}=\hat{z}_{i}$ is a nonvanishing
evaluation of $f'$. If a simple component $A_{1}$ say, is not
represented, then the same simple component is not represented in
the evaluation $x_{1}=\hat{z}_{1}\cdots \hat{z}_{t},
x_{2}=\hat{x}_{2},\ldots, x_{n}=\hat{x}_{n}$ and hence $f$
vanishes. We see that $f'$ vanishes on any evaluation which misses
a simple component.

We now turn to the proof of the $2$nd part of the lemma. If
$f'=\sum_{i}g_{i}\cdot f\cdot h_{i}$ then it is clear that if an
evaluation of $f'$ has less than $n_{A}-1$ radical evaluations
then with that evaluation $f$ has less than $n_{A}-1$ radical
evaluations and hence vanishes. This implies the vanishing of
$f'$. If an evaluation of
$f'(z_{1},\ldots,z_{t},x_{2},\ldots,x_{n})=f(Z,x_{2},\ldots,x_{n})$
has less than $n_{A}-1$ radicals, then this corresponds to an
evaluation of $f(x_{1},\ldots,x_{n})$ with less than $n_{A}-1$
radicals and hence vanishes.

\end{proof}

\begin{corollary}\label{Kemer polynomials of A are Phoenix}

Let $A$ be a basic algebra, then Kemer polynomials satisfy the Phoenix property.

\end{corollary}

\begin{proof}

Let $f$ be a Kemer polynomial of $A$ and let $f' \in \langle f \rangle$ be a multilinear nonidentity of $A$. We need to show there exists a Kemer polynomial in $\langle f' \rangle$.

We know that $A$ is full and satisfies property $K$ with respect
to the Kemer polynomial $f$. Furthermore, applying the previous
lemma, we have that $A$ is full and has property $K$ with respect
to $f'$. Finally, invoking Kemer's Lemma $2$ there exists $f'' \in
\langle f' \rangle$ which is Kemer.

\end{proof}

\end{section}

\begin{section}{Finite generation of the relatively free algebra}{\label{Finite-generation-of}}

Recall that an algebra $W$ satisfies the $m$th Capelli identity if any multilinear polynomial having an alternating set of cardinality
(at least) $m$ is an identity of $W$. The purpose of this section
is to prove that for any such algebra one can assume that the corresponding
relatively free algebra $\mathcal{W}$ is generated by (only) $m-1$
variables. More precisely, we will show that if
\[
\mathcal{W}=F\left\langle x_{1},\ldots,x_{m-1}\right\rangle /\Id(W)\cap F\left\langle x_{1},\ldots,x_{m-1}\right\rangle
\]
then $\Id(W)=\Id(\mathcal{W})$. To this end we recall some basic
results (and fix notation) from the representation theory of
$S_{n}$ (the symmetric group on $n$ elements) and their
application to PI theory.

Since $FS_{n}$ is a semisimple algebra, we can write $FS_{n}$ as
a direct sum of its minimal two sided ideals. It is a basic fact that
the number of such is equal to the number of conjugacy classes of
$S_{n}$. Each conjugacy
class can be described by a partition $\mu$ of $n$, i.e. a finite
sequence of nondecreasing natural numbers which sum up to $n$. As we shall see, each partition encodes the
structure of some minimal ideal.

Suppose $\mu=(\mu_{1},\ldots,\mu_{m})$ is a partition of $n$. A Young
diagram of $\mu$ is a finite subset of
$\mathbb{Z}\times\mathbb{Z}$ defined as $D_{\mu}=\{(i,j)|\,
i=1...m,\, j=1,\ldots,\mu_{i}\}$. This is not more than $m$ rows of
boxes, such that the length of the $i$'th row is $\mu_{i}$. A
Young tableau $T_{\mu}$ associated to $\mu$ is a filling of the
boxes of $D_{\mu}$ with the integers $1,\ldots,n$ (repetitions are
not allowed). We say that a tableau $T_{\mu}$ is \emph{standard}
if the numbers in each row and column of $T_{\mu}$ are increasing
from left to right and from up to bottom. To each tableau we
associate two subgroups of $S_{n}$ as follows: Let
$R_{T_{\mu}(1)},\ldots,R_{T_{\mu}(m)}$ denote the rows of $T_{\mu}$
(i.e. the numbers appearing in each row) and
$C_{T_{\mu}(1)},\ldots,C_{T_{\mu}(t)}$ denote the columns of
$T_{\mu}$. We denote by $\mathcal{R}_{T_{\mu}(i)}$ and by
$\mathcal{C}_{T_{\mu}(j)}$ the symmetric groups
$S_{R_{T_{\mu}(i)}}$ and $S_{C_{T_{\mu}(j)}}$ respectively (i.e
the symmetric groups acting on the numbers in the $i$th row and
$j$th column of $T_{\mu}$ respectively). Finally, we denote by
$\mathcal{R}_{T_{\mu}}$ and by $\mathcal{C}_{T_{\mu}}$ the
subgroups of $S_{n}$ which are the row and column stabilizers of
$T_{\mu}$. Clearly, with notation above, we have
$\mathcal{R}_{T_{\mu}}=\mathcal{R}_{T_{\mu}(1)}\times\cdots\times
\mathcal{R}_{T_{\mu}(m)}$ and
$\mathcal{C}_{T_{\mu}}=\mathcal{C}_{T_{\mu}(1)}\times\cdots\times
\mathcal{C}_{T_{\mu}(t)}$.

For each tableau $T_{\mu}$, consider the left ideal $V_{\mu}=FS_{n}e_{T_{\mu}}$, where
\[
e_{T_{\mu}}=\sum_{\sigma\in \mathcal{R}_{T_{\mu}},\tau\in \mathcal{C}_{T_{\mu}}}(-1)^{\tau}\sigma\tau.
\]

Now let $I$ be a minimal $2$-sided ideal of $FS_{n}$. It is well known that $I$ is equal to the sum of all minimal left
ideals isomorphic to some (fixed) minimal left ideal $V$.

\begin{theorem}

The following hold:
\begin{enumerate}

\item Let $V$ be a minimal left ideal of $FS_{n}$. Then there exists a partition $\mu$ such that $V\cong V_{\mu}$.
\item If we denote the corresponding minimal two sided ideal by $I_{\mu}$, then it is the direct sum of minimal ideals $V_{\mu}$ which correspond to standard
tableaux $T_{\mu}$ of $\mu$.
\item The map $T_{\mu} \mapsto V_{\mu}$ establishes a one to one correspondence between Young tableaux associated to Young diagram $D_{\mu}$ and minimal left
ideal isomorphic to $V_{\mu}$.
\end{enumerate}

\end{theorem}

Now let us apply the theorem above to PI theory. Let
$P_{n}(W)=P_{n}/(P_{n}\cap \Id(W))$, where $P_{n}$ is the space (of
$\dim_{F}(P_{n})=n!$) of all multilinear polynomials with
variables $x_{1},\ldots,x_{n}$. The group $S_{n}$ acts on $P_{n}(W)$
via $\sigma\cdot\overline{x_{i_{1}}\cdots
x_{i_{n}}}=\overline{x_{\sigma(i_{1})}\cdots x_{\sigma(i_{n})}}$
and hence we may consider its decomposition into irreducible
submodules. By the theorem above, any such submodule can be
written as $FS_{n}e_{T_{\mu}}\cdot f$, where $f$ is some
polynomial in $P_{n}(W)$. Clearly, if $f\in P_{n}(W)$ is nonzero,
then there is some partition $\mu$ and a (standard) tableau
$T_{\mu}$such that $e_{T_{\mu}}\cdot f$ is nonzero.

We are ready to prove the main result of this section.

\begin{theorem}
Let $W$ be an algebra which satisfies the $m$th Capelli identity. Then
$\Id(\mathcal{W})=\Id(W)$ where $\mathcal{W}$ is the relatively free algebra of $W$ generated by $m-1$
variables.
\end{theorem}

\begin{proof}
It is clear that $\Id(W)\subset \Id(\mathcal{W})$. For the other direction
suppose $f$ is a multilinear nonidentity of $W$ of degree $n$. Then, by the theorem above,
there is a partition $\mu$ of $n$ and a tableau $T_{\mu}$ such that $g=e_{T_{\mu}}\cdot f$
is a nonidentity of $W$.

Let $g_{0}=\sum_{\tau\in
\mathcal{C}_{T_{\mu}}}(-1)^{\tau}\tau\cdot f=\sum_{\tau\in
\mathcal{C}_{T_{\mu}(1)}}(-1)^{\tau}\tau\cdot\left(\sum_{k=1}^{l}(-1)^{\tau_{k}}\tau_{k}\cdot
f\right)$, where $\tau_{1},\ldots,\tau_{l}$ is a full set of
representatives of $\mathcal{C}_{T_{\mu}(1)}$-cosets in
$C_{T_{\mu}}$.

Let $h_{\mu}$ (the height of $\mu$) denote the number of rows
in the Young diagram $D_{\mu}$. If $h_{\mu} \geq m$, the
polynomial $g_{0}$ is alternating on the variables of the first
column and hence by assumption is an identity of $W$. But in that
case also the polynomial $g=\sum_{\sigma\in
R_{T_{\mu}}}\sigma\cdot g_{0}$ is in $\Id(W)$ contradicting our
assumption and so $h_{\mu}$ must be smaller than $m$.

Let us now focus on the rows of $D_{\mu}$. Since
$g=\sum_{\sigma\in \mathcal{R}_{T_{\mu}}}\sigma\cdot g_{0}$, it is
symmetric in the variables corresponding to any row of $T_{\mu}$
and so if for any $i=1,\ldots, h_{\mu}$ we replace by $y_{i}$
\textit{all} variables in $g$ corresponding to the $i$th row we
obtain a polynomial $\hat{g}$ which yields $g$ by
multinearization. In particular $g\in \Id(W)$ if and only if
$\hat{g}\in \Id(W)$. In order to conclude, note that $\hat{g}$ can
be regarded as an element of $\mathcal{W}$ (at most $m-1$ variables) and nonzero. This shows
that $g$ is nonidentity of $\mathcal{W}$ and hence also $f$. This
proves the theorem.
\end{proof}

\begin{remark}
In the sequel, if $W$ satisfies the $m$th Capelli identity, we'll consider affine relatively free algebras $\mathcal{W}$ with at least $m-1$ generating variables.

\end{remark}

\begin{definition}
Suppose $W$ is an affine algebra. Any algebra of the form
$$
F\left\langle x_{1},\ldots,x_{m}\right\rangle /\Id(W)\cap F\left\langle x_{1},\ldots,x_{m}\right\rangle
$$
having the same $T$-ideal as $W$ is called \emph{affine relatively
free algebra of $W$}.
\end{definition}

We close this section with the following useful lemma.

\begin{lemma}\label{useful lemma}
Let $W$ be a PI algebra which satisfies the $m$th Capelli
polynomial and let $\mathcal{W}=F\langle
x_{1},\ldots,x_{n}\rangle/\widehat{\Id(W)}$ be an affine
relatively free algebra where $n \geq m-1$. Let $I$ be any
$T$-ideal and denote by $\widehat{I}$ the ideal of $\mathcal{W}$
generated (or consisting rather) by all evaluation on
$\mathcal{W}$ of elements of $I$. Then
$\Id(\mathcal{W}/\widehat{I})= \Id(W) + I$.

\begin{proof}
Clearly,
$\Id(\mathcal{W}/\widehat{I}) \supseteq \Id(W)$ and
$\Id(\mathcal{W}/\widehat{I}) \supseteq I$ so
$\Id(\mathcal{W}/\widehat{I})$ contains $\Id(W) + I $. For the
converse note that

$$\mathcal{W}/\widehat{I}=(F\langle x_{1},\ldots,x_{n} \rangle)/\widehat{\Id(W)}/\widehat{I}=F\langle x_{1},\ldots,x_{n} \rangle/(\widehat{\Id(W)}+\widehat{I})=F\langle
x_{1},\ldots,x_{n}\rangle /\widehat{\Id(W)+I}.$$

Then, since $F\langle x_{1},\ldots,x_{n}\rangle
/\widehat{\Id(W)+I}$ satisfies the $n$th Capelli identity we obtain

$$\Id(F\langle x_{1},\ldots,x_{n}\rangle /\widehat{\Id(W)+ I})= \Id(W)+ I$$
as desired.

\end{proof}

\end{lemma}

\end{section}

\begin{section}{\label{Shirshov-base}Shirshov base}
\begin{definition}
Let $W$ be an affine algebra over $F$. Let $a_{1},\ldots,a_{s}$ be
a generating set of $W$. Let $m$ be a positive integer and let
$Y$ be the set of words in $a_{1},\ldots,a_{s}$ of length $\leq m$.
We say that $W$ has a Shirshov base of length $m$ and of height
$h$ if $W$ is spanned (over $F$) by elements of the form $y_{1}^{n_{1}}\cdots y_{l}^{n_{l}}$,
where $y_{i}\in Y$ and $l\leq h$.
\end{definition}

The following fundamental theorem was proved by Shirshov.
\begin{theorem}
If an affine algebra $W$ has a multilinear PI of degree $m$, then
it has a Shirshov base of length $m$ and some height $h$ where $h$ depends
only on $m$ and the number of generators of $W$.
\end{theorem}

In fact, there is an important special case where we can get even ``closer''
to representability.

\begin{lemma}\label{finite-module}
Let $C$ be a commutative algebra over $F$ and let $W=C\left\langle a_{1},\ldots,a_{s}\right\rangle $.
Suppose $W$ has a Shirshov base. If for every $i=1,\dots,s$, the element $a_{i}$ is
integral over $C$, then $W$ is a finite module over $C$.
\end{lemma}

If in addition, our commutative algebra $C$ is Noetherian and unital we reach our goal, as the next theorem
shows.

\begin{theorem} $($Beidar \cite{Bei}$)$\label{representable-over-commutative}
Let $W$ be an $F$ algebra and let $C$ be a unital commutative Noetherian $F$ algebra. If $W$
is a finite module over $C$, then $W$ is representable.
\end{theorem}
\begin{proof}
The simplest case is when $C$ is a local Artinian ring and $W$ is
of finite length (over $C$). Let $P$ be the unique maximal ideal
of $C$. Since $C$ is Artinian we know that $P$ is nilpotent, thus
$C$ is complete in the $P$-adic metric. Thus $C$ contains a field
$(F\subset)K$ such that $K+P=C$, and so $K$ is isomorphic to
$C/P$. Let $W=W_{(0)}\supseteq W_{(1)}\supseteq\cdots\supseteq
W_{(m)}=0$ be a composition series of $W$ (as a $C$ module). We
have that $W_{(i)}/W_{(i+1)}$ is isomorphic to $C/P$, so it is a
one dimensional $K$-space. It follows that $W$ is a finite
dimensional $K$-space.

Now assume that $W$ has finite length as a $C$ module (without
any other assumptions on the ring $C$). This yields that $C'=C/ann_{C}(W)$
is Artinian. Therefore, we may decompose $C'$ into a (finite) direct
product of local Artinian rings $C'=\times_{i=1}^{m}C_{i}$. This
decomposition induces a decomposition of $W$, namely $W=\oplus_{i=1}^{m}C_{i}W$, and since
each $C_{i}W$ is a $C_{i}$ module of finite
length, we are back to the first case (Note that here we make use of the
fact that a finite number of fields with the same characteristic can
be embedded in a larger field).

We proceed now to the general case. Recall that the set of associated
primes $Ass_{C}(W)$ is a finite set. Let $P_{1},\ldots,P_{k}$ be its
maximal elements and let $S=C-\cup_{i=1}^{k}P_{i}$ (a multiplicative
set). Since $S$ is the set of nonzero divisors of $W$, the localization
by $S$, $C_{0}=S^{-1}C$, induces an embedding of $W$ in $W_{0}=S^{-1}W$.
Moreover, the maximal ideals of $C_{0}$ are $S^{-1}P_{1},\ldots,S^{-1}P_{k}$
and $Ass_{C_{0}}(W_{0})=\{S^{-1}P|\, P\in Ass_{C}(W)\}$ (the last
equality is due to the maximality of the $P_{i}$'s in $Ass_{C}(W)$).
Thus we may assume from the beginning that $C$ is a semilocal ring
whose maximal ideals are contained in $Ass_{C}(W)$.

We continue by induction on $d_{C}(W)$, where $d_{C}(W)$ represents
the greatest length of a descending chain of prime ideals from $Ass_{C}(W)$.
If $d_{C}(W)=1$, then all the ideals of $Ass_{C}(W)$ are the maximal
ideals of $C$. Since $Ass_{C}(W)$ contains all minimal primes containing
$ann_{C}(W)$, we conclude that all primes which contain $ann_{C}(W)$
are maximal. Hence $W$ is of finite length and we are done.

Suppose now that $d_{C}(W)>1$. Let $Q_{1},\ldots,Q_{l}$ be all maximal
elements in the set $H=Ass_{C}(W)-\{P_{1},\ldots,P_{k}\}$ and denote
$T=C-\cup Q_{i}$. Write $C_{1}=T^{-1}C,\, W_{1}=T^{-1}W$ and let
$U$ be the kernel of the canonical map $W\rightarrow W_{1}$. It
is easy to calculate the associated primes:
\[
Ass_{C_{1}}(W_{1})=\{T^{-1}P|\, P\in H\},\, Ass_{C}(U)=\{P_{1},\ldots,P_{k}\}
\]
Since $d_{C_{1}}(W_{1})<d_{C}(W)$, we know by the induction hypothesis
that $W_{1}$ is representable.

There is $r>0$ for which $J^{r}U=0$, where $J=\cap P_{i}$ is the
Jacobson radical of $C$. Using the Artin-Rees lemma we obtain some
$r'$ for which $J^{r'}W\cap U=0$. Thus, $W$ is a sub direct product
of $W/U$ and $W/J^{r'}W$. Finally, these two algebras are representable:
$W/U$ - since it is contained in $W_{1}$, and $W/J^{r'}W$ - since
it is of finite length (recall that there is a product of primes which
annihilate it).
\end{proof}

\end{section}

\begin{section}{\label{The-trace-ring}The trace ring}

Suppose $A$ is an $M$-dimensional $F$-algebra and let
$\mathcal{A}$ be an affine relatively free algebra, say generated
by the elements $\{x_1+\Id(A),\ldots,x_v+\Id(A)\}$, where
$x_1,\ldots,x_v$ are noncommuting variables.

It is well known that $\mathcal{A}$ may be interpreted as an
algebra of ``generic elements''. Let us recall briefly the
construction.

Let $K=F\left(\{t_{i,j}|\, i=1,\ldots,v\,\,; j=1,\ldots,M\}\right)$ and
suppose $a_{1},\ldots,a_{M}$ is an $F$-basis of $A$. Consider the
$F$-subalgebra $\mathcal{A}'$ of $A_{K}=A\otimes_{F}K$ generated
by
\[
y_{i}=\sum_{j=1}^{M}t_{i,j}a_{j}, i=1,\ldots,v.
\]
It is well known that the map
$\phi:\mathcal{A}'\rightarrow\mathcal{A}$ determined by
$\phi(y_{i})=x_{i}+\Id(A)$, is an $F$-algebra isomorphism.
Henceforth we will not distinguish between $\mathcal{A}$ and
$\mathcal{A}'$ and we denote both by $\mathcal{A}$. It follows
from the construction of the algebra of generic elements that
$\mathcal{A}$ is representable and so we fix for the rest of the
paper an embedding $\mathcal{A} \subseteq A_{K}$.

Suppose $A=A_{1}\times\cdots\times A_{k}$, where $A_{r}$ is finite
dimensional over $F$, for $r=1,\ldots,k$. Fix a $K$-vector space
decomposition $(A_{r})_{K} =(A_{r})_{K}^{ss}\oplus (J_{r})_{K}$
where $(A_{r})_{K}^{ss}$ is a maximal semisimple subalgebra of
$(A_{r})_{K}$ and $(J_{r})_{K}$ is its Jacobson radical. Consider
the embedding

\[
(A_{r})_{K}^{ss}\hookrightarrow End_{K}((A_{r})_{K}^{ss})
\]
given by $\ensuremath{a(x)=ax}\ensuremath{(a,x\in (A_{r})_{
K}^{ss}})$. Let $\hat{\mathcal{A}}$ be the $F$-subalgebra of
$A_{K}$ generated by the projections of $\mathcal{A}$ into
$A_{K}^{ss}=(A_{1})_{K}^{ss}\times\cdots\times (A_{k})_{K}^{ss}$
and $J_{K}=(J_{1})_{K}\times\cdots\times (J_{k})_{K}$. It is clear
that $\hat{\mathcal{A}}$ is an affine $F$-algebra containing (the
image of) $\mathcal{A}$. Moreover, we may choose generators of
$\hat{\mathcal{A}}$ such that the corresponding Shirshov base is
$Y=Y_{ss}\sqcup Y_{J},$ where $Y_{ss}\subset A_{K}^{ss}$ and
$Y_{J}\subset J_{K}$ (indeed, choosing generators
$b_{1},\ldots,b_{s}$ of $\hat{\mathcal{A}}$ either from
$A_{K}^{ss}$ or $J_{K}$, then $b_{i_{1}}b_{i_{2}}\cdots
b_{i_{t}}\in Y_{ss}$, $t\leq m$, if and only if $b_{i_{j}}\in
A_{K}^{ss}$, for all $j=1,\ldots,t$). Since $J_{K}$ is nilpotent,
it is clear that the elements of $Y_{J}$ are integral over $F$,
however this may not be the case for elements in $Y_{ss}$. Our
goal is to extend $F$ in a suitable way so that the elements of
$Y_{ss}$ ``become'' integral. Consider the elements of $Y_{ss}$ as
elements of $End_{K}((A_{1})_{K}^{ss})\times\cdots\times
End_{K}((A_{k})_{K}^{ss})$ (via the above embedding) and define $R$
to be the unital $F$-subalgebra of $K\times\cdots\times K$ ($k$
times) generated by
$Tr(Y_{ss})=\{\left(Tr(p_{1}(u))),\ldots,Tr(p_{k}(u)): u \in
Y_{ss}\right)\}$, where $p_{r}$ is the projection
$End_{K}((A_{1})_{K}^{ss})\times\cdots\times End_{K}((A_{k})_{K}^{ss})
\rightarrow End_{K}((A_{r})_{K}^{ss})$. Since $Y_{ss}$ is a finite
set, it is clear that $R$ is a commutative unital Noetherian
$F$-algebra.

Let us define an action of $K\times\cdots\times K$ ($k$-times)
(and hence also of $R$) on $A_{K}=(A_{1})_{K}\times\cdots\times
(A_{k})_{K}$ by
\[
(c_{1},\ldots,c_{k})\cdot(b_{1},\ldots,b_{k})=(c_{1}b_{1},\ldots,c_{k}b_{k}),
\]
where $(c_{1},\ldots,c_{k}) \in K^{(k)}$ and $b_{1}\in
(A_{1})_{K},\ldots,b_{k}\in (A_{k})_{K}$.

Let $d_{r}=dim_{F}A_{r}^{ss}$ and denote by
$d=max\{d_{r}:r=1,\ldots,k\}$.

\begin{theorem}
$\mathcal{A}_{R}=R\mathcal{\cdot A}$ is a finite module over $R$.
\end{theorem}
\begin{proof}
The Cayley-Hamilton theorem implies that for $a\in Y_{ss}$

\[
p_{r}(a^{d+1})+\sum_{i=0}^{d-1}q_{i}\left(Tr(p_{r}(a^{1})),\ldots,Tr(p_{r}(a^{d}))\right)p_{r}(a^{i+1})=0,
\]
where $q_{0},\ldots,q_{d-1}$ are polynomials on commutative
variables.

It follows that
\[
a^{d+1}+\sum_{i=0}^{d-1}q_{i}\left(Tr(a^{1}),\ldots,Tr(a^{d})\right)a^{i+1}=0
\]
proving that $Y_{ss}$, and hence also $Y$, is integral over $R$.

\begin{remark}
The exponent $d+1$ is needed since the generic algebra
$\mathcal{A}$ has no identity.

\end{remark}

Applying Lemma \ref{finite-module} we obtain that $R\cdot\hat{\mathcal{A}}$ is
a finite module over $R$. Since $R$ is Noetherian,
$R\cdot\hat{\mathcal{A}}$ is a Noetherian $R$-module, and hence
$\mathcal{A}_{R}\subseteq R\cdot\hat{\mathcal{A}}$ is a finite
module over $R$. This proves the theorem. \end{proof}
\begin{corollary}\label{Representability of relatively free modulo an ideal}
The algebra $\mathcal{A}_{R}$ is representable. Furthermore, if
$I$ is an ideal of $\mathcal{A}$ which is closed under
multiplications by elements of $R$, that is $I_{R}=RI=I$, then
$\mathcal{A}/I$ is representable. \end{corollary}
\begin{proof}
Since $\mathcal{A}_{R}$ is a finite module over $R$, where $R$ is
unital, commutative and Noetherian, it is representable (by
Theorem \ref{representable-over-commutative}). For the same reason, if $I$ is any ideal of
$\mathcal{A}$, the quotient module $\mathcal{A}_{R}/I_{R}$ (being
finite) is representable. Now suppose $I_{R}=I$. Then we get
$\mathcal{A}/I\subseteq\mathcal{A}_{R}/I=\mathcal{A}_{R}/I_{R}$.
Since the later is representable, the result follows.
\end{proof}

We would like to apply the above results to the case where
$A=A_{1}\times\cdots\times A_{k}$ is a product of \emph{basic}
algebras but before that we need the following basic lemma.

\begin{lemma}\label{generation of evaluation ideal}
Let $F\langle X \rangle$ be the free algebra over $F$ where $X$ is
a countable set of variables. Let $W$ be any algebra and $S$ a set
of polynomials in $F\langle X \rangle$. Let $I = \langle S
\rangle$, the $T$-ideal generated by $S$. Denote by
$\mathcal{I}$ and $\mathcal{S}$ the sets of \textit{all}
evaluations on $W$ of polynomials of $I$ and $S$ respectively.
Then $\mathcal{I} = \langle \mathcal{S} \rangle$ (the ideal
generated by $\mathcal{S}$).

\end{lemma}
\begin{proof}
We show first $\mathcal{I}$ is an ideal of $W$. Let $z,w \in
\mathcal{I}$ and let $p_{z}$ and $p_{w}$ be polynomials in
$I$ with evaluations $z$ and $w$ respectively. By the
$T$-property of $I$ we may change variables and so we
may assume $p_{z}$ and $p_{w}$ have disjoint sets of variables.
Then there is an evaluation of $p_{z} + p_{w}$ which is $z+w$.
Next, let $z \in \mathcal{I}$ and $u \in W$. If $p_{z} \in
I$ with value $z$, we may take a variable $x$ which is
not in $p_{z}$ and get $uz$ as an evaluation of $xp_{z}$.

Now, obviously $\mathcal{I} \supseteq \langle \mathcal{S}
\rangle$. For the converse, consider the algebra
$\overline{W}=W/\langle \mathcal{S} \rangle$. Clearly, elements of
$S$ are identities of $\overline{W}$ and hence $I = \langle S
\rangle \subseteq \Id(\overline{W})$. It follows that all
evaluations of $I$ on $W$ are contained in $\langle \mathcal{S}
\rangle$ as desired.
\end{proof}

\begin{proposition}\label{closed under multiplication}
Let $A$ be a product of basic algebras as above and let $I$ be the
$T$-ideal generated by \emph{some} Kemer polynomials of $A$.
Denote by $\hat{I}$ the ideal of $\mathcal{A}$ obtained by all
evaluation of the polynomials of $I$ on $\mathcal{A}$. Then
$\hat{I}$ is closed under multiplication of $R$.\end{proposition}
\begin{proof}
We need to show that for any element $f\in I$, any evaluation
$\bar{f}\in\hat{I}$ (of $f$ on $\mathcal{A}$) and any $a_{0}\in
Y_{ss}$, we have $Tr(a_{0})\bar{f}\in\hat{I}$. Note that since the
ideal $\hat{I}$ is generated by all evaluations on $\mathcal{A}$
of all Kemer polynomials in $I$, invoking Lemma \ref{generation of evaluation ideal}, we may assume that $f$ is a Kemer
polynomial.

Let $(d,s)$ be the Kemer index of $A$. Recall that $d=max\{d_r\}$
where $d_{r}=dim_{F}(A_{r}^{ss})$ and if $\Psi=\{q:d_q=d\}$ then
$s=max_{q\in \Psi}\{n_{A_{q}}-1\}$, where $n_{A_{q}}$ is the
nilpotency index of $J_{q}$. Let
$f(Z,X_{1},\ldots,X_{\mu},V_{1},\ldots,V_{s},Q)$ be a
(multilinear) Kemer polynomial of $A$ where
$Z=\{z_{1},\ldots,z_{d}\},X_{1},\ldots,X_{\mu}$ are small sets,
$V_{1},\ldots,V_{s}$ are big sets (the designated variables) and
$Q$ is a set of additional variables. We assume $\mu$ is large
enough so that $f$ is an identity of any basic algebra $A_{r}$
whose Kemer index is strictly smaller than $(d,s)$ (note that since $f$ is a Kemer polynomial of $A$, it is a nonidentity of $A_q$ for some $q \in \Psi$ and $n_{A_q}=s$). Consider an
evaluation of $f$ (on $\mathcal{A}$) given by
$\hat{Z}=\{\hat{z}_{1},\ldots,\hat{z}_{d}\}$,
$\{\widehat{X}_{i}\}$, $\{\widehat{V}_{j}\}$, $\widehat{Q}$.

In view of the embedding
\[
\mathcal{A}\subseteq\hat{\mathcal{A}}\subseteq
A_{K}=\prod_{r=1}^{k}(A_{r})_{K}^{ss}\oplus (J_{r})_{K}
\]
each element
$\pi\in\widehat{Z}\cup(\cup_{i}\widehat{X}_{i})\cup(\cup_{j}\widehat{V}_{j})\cup\widehat{Q}$
can be written as $\pi=\pi^{ss}+\pi^{J}$ where $\pi^{ss}\in
A_{K}^{ss}$ and $\pi^{J}\in J_{K}$. Note that in general
$\pi^{ss}$ and $\pi^{J}$ are not elements of $\mathcal{A}$. Now,
viewing the above evaluation in $A_{K}$, we have ($t$ denotes the
degree of $f$)

\[
f(\pi_{1},\ldots,\pi_{t})=f\left(\pi_{1}^{ss}+\pi_{1}^{J},\ldots,\pi_{t}^{ss}+\pi_{t}^{J}\right)=\sum_{i=1}^{2^{t}}\bar{f}_{i}
\]
where
$\bar{f}_{i}=f(\pi_{1}^{\epsilon_{1}},\ldots,\pi_{t}^{\epsilon_{t}})$
and $\pi^{\epsilon_{j}}$ is either $\pi^{ss}$ or $\pi^{J}$.

Since $A_{K}$ has the same Kemer index as $A$, the evaluation
$\bar{f}_{i}$ (on $A_{K}$) vanishes unless all small sets, $Z$,
$\{X_{i}\}_{i}$ get semisimple values (and precisely one variable
of each big set $\{V_{1},\ldots,V_{s}\}$ gets a radical value). In
particular we have

\[
\bar{f}=f(\widehat{Z},\{\widehat{X}_{i}\},\{\widehat{V}_{j}\},\widehat{Q})=\sum_{r=1}^{k}f(\hat{z}_{1}^{ss}\ldots,\hat{z}_{d}^{ss},\{\widehat{X}_{i}\},\{\widehat{V}_{j}\},\widehat{Q}),
\]
where $\hat{z}_{i}=\hat{z}_{i}^{ss}+\hat{z}_{i}^{J}$. Our task is
then to show
$Tr(a_{0})f(\hat{z}_{1}^{ss}\ldots,\hat{z}_{d}^{ss},\{\widehat{X}_{i}\},\{\widehat{V}_{j}\},\widehat{Q})\in\hat{I}.$
To this end let us simplify the notation. We let
$a_{1}=\hat{z}_{1}^{ss},\ldots,a_{d}=\hat{z}_{d}^{ss}$,
$b_{1}=\hat{z}_{1}^{J},\ldots,b_{d}=\hat{z}_{d}^{J}$ (that is
$\hat{z}_{i}=a_{i}+b_{i}$, $i=1,\ldots,d$) and since the tuples
$\{\widehat{X}_{i}\},\{\widehat{V}_{j}\},\widehat{Q}$ will not
play any role in the proof we denote them by $\Lambda$. Thus
$\bar{f}=f(a_{1},\ldots,a_{d},\Lambda)$.

Claim: for any $r$ the following holds:
\[
Tr(a_{0}^{(r)})f(a_{1}^{(r)},\ldots,a_{d}^{(r)},\Lambda^{(r)})=\sum_{i=1}^{d}f(a_{1}^{(r)},\ldots,a_{i-1}^{(r)},a_{0}^{(r)}a_{i}^{(r)},a_{i+1}^{(r)},\ldots,a_{d}^{(r)},\Lambda^{(r)}),
\]
where
$a_{i}=(a_{i}^{(1)},\ldots,a_{i}^{(k)})=(p_{1}(a_{i}),\ldots,p_{k}(a_{i}))$
and $\Lambda^{(r)}=p_{r}(\Lambda)$.

Since the variables $z_{1},\ldots,z_{d}$ alternate in $f$, the
value $f(a_{1}^{(r)},\ldots,a_{d}^{(r)},\Lambda^{(r)})$ is zero
unless the elements $a_{1}^{(r)},\ldots,a_{d}^{(r)}$ are linearly
independent over $K$. On the other hand, since $A_{r}$ is basic,
$d_{r}=dim_{F}(A_{r}^{ss})=dim_{K}((A_{r})_{K}^{ss}\leq d$ and so
$f(a_{1}^{(r)},\ldots,a_{d}^{(r)},\Lambda^{(r)})\neq 0$ only if
$d_{r}=d$ and the set $\{a_{1}^{(r)},\ldots,a_{d}^{(r)}\}$ is a
basis of $(A_{r})_{K}^{ss}$ over $K$.

Write $a_{0}^{(r)}a_{i}^{(r)}=\sum_{j=1}^{d}\gamma_{i,j}a_{j}$
where $\gamma_{i,j}\in K$.

We compute:
\[
\sum_{i=1}^{d}f(a_{1}^{(r)},\ldots,a_{i-1}^{(r)},a_{0}^{(r)}a_{i}^{(r)},a_{i+1}^{(r)},\ldots,a_{d}^{(r)},\Lambda^{(r)})=\sum_{i=1}^{d}\sum_{j=1}^{d}\gamma_{i,j}f(a_{1}^{(r)},\ldots,a_{i-1}^{(r)},a_{j}^{(r)},a_{i+1}^{(r)},\ldots,a_{d}^{(r)},\Lambda^{(r)}).
\]
 As $f(a_{1}^{(r)},\ldots,a_{i-1}^{(r)},a_{j}^{(r)},a_{i+1}^{(r)},\ldots,a_{d}^{(r)},\Lambda^{(r)})$
is zero when $a_{j}^{(r)}\neq a_{i}^{(r)}$, we obtain
\[
\sum_{i=1}^{d}\gamma_{i,i}f(a_{1}^{(r)},\ldots,a_{i-1}^{(r)},a_{i}^{(r)},a_{i+1}^{(r)},\ldots,a_{d}^{(r)},\Lambda^{(r)})=Tr(a_{0}^{(r)})f(a_{1}^{(r)},\ldots,a_{d}^{(r)},\Lambda^{(r)})
\]
proving the claim.

Having established the claim for $r=1,\ldots,k$ we conclude

\[
Tr(a_{0})\cdot
f(a_{1},\ldots,a_{d},\Lambda)=\sum_{i=1}^{d}f(a_{1},\ldots,a_{i-1},a_{0}a_{i},a_{i+1},\ldots,a_{d},\Lambda).
\]

It order to complete the proof we show that
$f(a_{1},\ldots,a_{i-1},a_{0}a_{i},a_{i+1},\ldots,a_{d},\Lambda)\in\hat{I}$,
$i=1,\ldots,d$. The argument is similar to the one above. Recall
that $a_{0}\in Y_{ss}\subseteq\hat{\mathcal{A}}$. Furthermore, by
the construction of $\hat{\mathcal{A}}$, there is an element
$\hat{z}_{0}\in\mathcal{A}$ such that $\hat{z}_{0}=a_{0}+b_{0}$,
where $b_{0}\in J_{K}$. This implies
\[
f(a_{1},\ldots,a_{i-1},a_{0}a_{i},a_{i+1},\ldots,a_{d},\Lambda)=f(\hat{z}_{1},\ldots,\hat{z}_{i-1},\hat{z}_{0}\hat{z}_{i},\hat{z}_{i+1},\ldots,\hat{z}_{d},\Lambda)
\]
 which is clearly in $\widehat{I}$. The proposition is now proved.\end{proof}

Proposition \ref{closed under multiplication} and Corollary \ref{Representability of relatively free modulo an ideal} yield

\begin{corollary} \label{representability of a relatively free algebra modulo the T-ideal of Kemer polynomials}
Let $A$ be a finite dimensional $F$-algebra and let $\mathcal{A}$
be an affine relatively free algebra. Let $I$ be a $T$-ideal
generated by some Kemer polynomials of $A$ and let $\hat{I}$ be
the ideal of $\mathcal{A}$ consisting of all evaluations of $I$ on
$\mathcal{A}$. Then $\mathcal{A}/\hat{I}$ is representable.

\end{corollary}

Let us explain how the last result fits in our plan for proving
Kemer's theorem. We have started the proof by showing the existence of a finite
dimensional $F$-algebra $A$ whose ideal of identities is contained
in $\Id(W)$. Our goal is to replace the algebra $A$ by a
representable algebra $A'$, with $\Id(A)\subseteq \Id(A')\subseteq
\Gamma$, which has the same Kemer index as  $\Gamma$, and moreover
shares with $\Gamma$ the same Kemer polynomials. This may be
considered as an ``approximation'' (i.e. not necessarily PI
equivalence) of $\Gamma$ by a representable algebra. Let us sketch
briefly the construction of $A'$. We take the $T$-ideal $I$
generated by Kemer polynomials of $A$ which are contained in
$\Gamma$. If $A$ and $\Gamma$ have different Kemer indices, then
all Kemer polynomials of $A$ are contained in $\Gamma$ and hence
the Kemer index of $\mathcal{A}/\hat{I}$ is strictly smaller than
the Kemer index of $A$. A finite number of such steps yield a
representable algebra $A_{1}$ with the same Kemer index as
$\Gamma$ and $\Id(A_{1}) \subseteq \Gamma$. Once we have reached
the Kemer index of $\Gamma$ (from above in the lexicographic
ordering), we consider the $T$-ideal $I$ generated by all Kemer
polynomials of $A_{1}$ which are contained in $\Gamma$. As above,
we let $\hat{I}$ be ideal of $\mathcal{A}_{1}$ (an affine
relatively free algebra of $A_{1}$) consisting of all evaluations
of $I$ on $\mathcal{A}_{1}$ and conclude our construction by
putting $A'= \mathcal{A}_{1}/\hat{I}$. Before we present the
details of the proof, let us explain why we insist in modding out
ideals of a relatively free algebra $\mathcal{A}$ (and not of $A$
for instance).

It is clear that if $B$ is any algebra and $I$ is a $T$-ideal,
modding out from $B$ the ideal $\hat{I}$ consisting of all
evaluations of $I$ on $B$, yields an algebra whose $T$-ideal of
identities contains $I$. However, in general, we don't know
whether other polynomials ``become'' identities. For instance if
$B=M_{n}(F)$, then taking any $T$-deal $I\nsubseteq \Id(B)$ gives
$B/\hat{I}=0$. A key property of the relatively free algebra
$\mathcal{A}$ is that $\Id(\mathcal{A}/\hat{I})=\Id(A)+I$ (see
Lemma \ref{useful lemma}).

\end{section}

\begin{section}{$\Gamma$-Phoenix property}

Suppose $\Gamma$ is a $T$-ideal containing a Capelli polynomial.
We know that this is equivalent to saying that $\Gamma$ is a
$T$-ideal of an affine PI algebra and also equivalent to $\Gamma$
containing the $T$-ideal of a finite dimensional algebra $A$. If
we denote by $p_{\Gamma}$ and $p_{A}$ the Kemer index of $\Gamma$
and $A$ respectively, then $p_{\Gamma}\leq p_{A}$. Our goal in
this section is to show that it is possible to replace $A$ by
another finite dimensional algebra $B$, with $\Id(A) \subseteq
\Id(B) \subseteq \Gamma$, which is ``closer'' to $\Gamma$ in the
sense that its Kemer index and Kemer polynomials are exactly as
those of $\Gamma$. This will allow us to deduce the Phoenix
property for Kemer polynomials of $\Gamma$ from (the already
established) Phoenix property for Kemer polynomials of basic
algebras (see Corollary \ref{Kemer polynomials of A are Phoenix}).

Let us recall our notation once again. Let $A$ be a finite
dimensional algebra which is a direct product of basic algebras
$A_{1}\times\cdots\times A_{s}$. Let $p_{A}$ and $p_{i}$ denote
the Kemer index of $A$ and $A_{i}$, $i=1,\ldots,s$ respectively.
We let $\mu_{i}$ the minimal number of small sets in its Kemer
polynomials. Finally, write $\mu_{0}$ for the maximum of
$\{\mu_{1},\ldots,\mu_{s}\}$. $\mathcal{A}$ denotes an affine relatively free algebra that corresponds to $A$.

In the next proposition $\Id(A) \subseteq \Gamma=\Id(W)$ where $W$
is an affine PI algebra over $F$ and $A$ is a finite dimensional
algebra over $F$.

\begin{proposition}\label{pre-phoenix}
There exists a representable algebra $B$ with the following
properties:
\begin{enumerate}
\item $\Id(B)\subseteq\Gamma$.
\item The Kemer index $p_{B}$ of $B$ coincides with $p_{\Gamma}$.
\item $\Gamma$ and $B$ have the same Kemer polynomials corresponding to
every $\mu$ which is $\geq\mu_{0}$.
\end{enumerate}
\end{proposition}
\begin{corollary}
Extending scalars to a larger field we may assume the algebra $B$
is finite dimensional over $F$.
\end{corollary}

\begin{proof}(of proposition)

Our first goal is to construct a representable algebra $B$ with
$\Id(B) \subseteq \Gamma$ and $p_{B}=p_{\Gamma}$. To this end, we
assume $p_{A} > p_{\Gamma}$. It follows that there exists
$\mu_{0}$, such that any Kemer polynomial of $A$ with at least
$\mu_{0}$ small sets is in $\Gamma$ (indeed, if there is no such
$\mu_{0}$, then for any $\mu_{0}$, there is a Kemer polynomial of
$A$ with $\mu_{0}$ small sets of cardinality $d$ and $s$ big sets
of cardinality $d+1$ (where $p_{A}= (d,s)$) which is not in
$\Gamma$. This says, by definition of the Kemer index, that $(d,s)
\leq p_{\Gamma}$, contrary to our assumption). Let $I$ be the
$T$-ideal generated by all Kemer polynomials of $A$ with at least
$\mu_{0}$ small sets and let $\hat{I}$ be the ideal of
$\mathcal{A}$ consisting of all evaluations of polynomials in $I$
on $\mathcal{A}$. Due to Corollary \ref{representability of a
relatively free algebra modulo the T-ideal of Kemer polynomials}
and Lemma \ref{useful lemma} we know $\mathcal{A}/\hat{I}$ is
representable and $\Id(\mathcal{A}/\hat{I})=\Id(A) + I \subseteq
\Gamma$. Let us show that $p_{\mathcal{A}/\hat{I}} < p_{A}$.
Clearly, $p_{\mathcal{A}/\hat{I}} \leq p_{\mathcal{A}}=p_{A}$.
Suppose $p_{\mathcal{A}/\hat{I}} = p_{A}$ and let $f$ be a Kemer
polynomial of $\mathcal{A}/\hat{I}$ with at least $\mu_{0}$ small
sets. Since $\Id(\mathcal{A}/\hat{I}) \supseteq \Id(A)$, the
polynomial $f$ is a nonidentity of $A$. It follows that $f$ is a
Kemer polynomial of $A$ and hence is in $I$. We obtain that $f \in
\Id(\mathcal{A}/\hat{I})$ contradicting our assumption on $f$. It
is clear that repeating the process above (a finite number of
times) we obtain a representable algebra $B$ with $\Id(B)
\subseteq \Gamma$ and $p_{B}=p_{\Gamma}$.

In order to complete the proof of the proposition let us assume we
have a finite dimensional algebra $A$ with $\Id(A) \subseteq
\Gamma$ and $p_{A}=p_{\Gamma}$. We need to construct a
representable algebra $B$ with $\Id(A) \subseteq \Id(B) \subseteq
\Gamma$ (and hence $p_{B}=p_{\Gamma})$ such that $B$ and $\Gamma$
have the same Kemer polynomials (with at least $\mu_{0}$ small
sets). Let $I$ be the $T$-ideal generated by all Kemer polynomials
of $A$ which are contained in $\Gamma$ and let $\hat{I}$ the
corresponding ideal of $\mathcal{A}$. Note that in this final
step, it is necessarily not true that all Kemer polynomials of $A$
are contained in $\Gamma$ since any Kemer polynomial of $\Gamma$
is a Kemer polynomial of $A$. Consider the algebra
$B=\mathcal{A}/\hat{I}$. We know $B$ is representable and has the
same Kemer polynomials as $\Gamma$.
\end{proof}

\begin{theorem}$(${$\Gamma$-Phoenix property of Kemer polynomials}$)$ \label{Phoenix-property}

Let $\Gamma$ be a $T$-ideal as above and let $f$ be a Kemer polynomial of $\Gamma$.
Then it satisfies the $\Gamma$-Phoenix property.
\end{theorem}
\begin{proof}
Let $\langle f \rangle$ be the $T$-ideal generated by $f$ and let
$h \in \langle f \rangle$ be a polynomial not in $\Gamma$. We need
to show there is $f' \in \langle h \rangle$ which is Kemer of
$\Gamma$. By the proposition there is a finite dimensional algebra
$A$ with $\Id(A) \subseteq \Gamma$ whose Kemer polynomials are
precisely those of $\Gamma$. Hence $f$ is a Kemer polynomial of
$A$ and assuming (as we may) that $A=A_{1}\times \cdots \times A_{s}$ where
$A_{i}$ are basic, the polynomial $f$ is Kemer of $A_{i}$ for some
$i$. But more than that, $f$ is Kemer with respect to each basic
algebra $A_{j}$ as long as $f \notin \Id(A_{j})$. Note that for any
such $j$, $p_{j}=p_{A}=p_{\Gamma}$. Now the polynomial $h$ is not
in $\Gamma$ and hence is not in $\Id(A)$. It follows that $h \notin
\Id(A_{j_{0}})$ for some $j_{0}$ showing that $f \notin \Id(A_{j_{0}})$.
As mentioned above, $f$ must be Kemer for $A_{j_{0}}$ and so we
may apply the Phoenix property for Kemer polynomials of the basic
algebra $A_{j_{0}}$ (see Corollary \ref{Kemer polynomials of A are Phoenix}). This says that there is $f' \in \langle h
\rangle$ which is Kemer for $A_{j_0}$ and hence Kemer for $A$.
Applying once again Proposition \ref{pre-phoenix} we have that $f'$ is
Kemer of $\Gamma$. The theorem is now proved.

\end{proof}

\end{section}

\section{Technical tools}

\begin{subsection}{Zubrilin-Razmyslov Traces}

We have seen already the usefulness of traces for the purpose of representability (via the
Cayley-Hamilton theorem). The
theme of this section is to get a version of this theorem to a
more general PI setting. We start by introducing the analogue of ``characteristic values'' (this terminology will be clearer below).

\begin{definition}
Let $f(x_{1},\ldots,x_{n},\Lambda)$ be a polynomial and let $z$ be any
variable where $z \neq x_{i}$, $i=1,\ldots,n$. We define the polynomial

\[ \delta_{k}^{z}|_{x_{1},\ldots,x_{n}}(f)=
\sum_{1\leq i_{1}<\cdots<i_{k}\leq n}f_{i_{1},\ldots,i_{k}}
\]
where $f_{i_{1},\ldots,i_{k}}$ is the polynomial obtained from $f$ by
substituting $zx_{i_{j}}$ in $x_{i_{j}}$ for $j=1,\ldots,k$. Using
Zubrilin's notation $f_{i_{1},\ldots,i_{k}}=f|_{zx_{i_{1}}\rightarrow
x_{i_{1}},\ldots,zx_{i_{k}}\rightarrow x_{i_{k}}}$.

For $k=0$, we set $\delta_{0}^{z}|_{x_{1},\ldots,x_{n}}(f)=f$.

\end{definition}
\begin{remark}
Notice that the operators $\delta_{k}^{z}$'s are $F$-linear. They depend on the variables
$x_{1},\ldots,x_{n}$, however since we always refer to the same
variables $x_{1},\ldots,x_{n}$ we adopt the abbreviated notation
$\delta_{k}^{z}$.
\end{remark}

\begin{remark}
Notice that this definition makes sense also in case $f$ does not
depend on some of the variables $x_{1},\ldots,x_{n}$ and in
particular, for polynomials $f$ which are free of these variables
(e.g.
$\delta_{4}^{z}|_{x_{1},\ldots,x_{6}}(x_1y_1x_2y_2)=\sum_{1\leq
i_{1}<i_2<i_3<i_{4}\leq 6}f_{i_{1},\ldots,i_{4}}=$. We have $15$
tuples $(i_1 < i_2 < i_3 < i_4)$
$$\{(1234), (1235), (1236), (1245),
(1246), (1256)$$ $$(1345), (1346), (1356), (1456)$$ $$(2345),
(2346), (2356), (2456)$$ $$(3456)\}.$$ This yields the following
polynomial:
$$
6zx_1y_1zx_2y_2 + 4zx_1y_1x_2y_2 + 4x_1y_1zx_2y_2.
$$

\end{remark}
\begin{remark}
We may consider the operation $\delta_{k}^{z}$ in the following
setting. Let $F\langle x_{1},\ldots,x_{n}, \Sigma \rangle$ be an
affine free algebra where $\Sigma$ is a finite set of variables.
Let $z=p(\Sigma)$ be any polynomial which is free of the variables
$x_{1},\ldots,x_{n}$. We define the operator $\delta_{k}^{p}$ on
$F\langle x_{1},\ldots,x_{n}, \Sigma \rangle$ by

\[
\delta_{k}^{p(\Sigma)}|_{x_{1},\ldots,x_{n}}(f)=\sum_{1\leq i_{1}<\cdots<i_{k}\leq n}f_{i_{1},\ldots,i_{k}}
\]
where $f_{i_{1},\ldots,i_{k}}$ is the polynomial obtained from $f$ by
substituting $p(\Sigma)x_{i_{j}}$ in $x_{i_{j}}$ for $j=1,\ldots,k$.
\end{remark}

\begin{remark}
Another way to view $\delta_{k}^{z}|_{x_{1},\ldots,x_{n}}(f)$, in
case $f$ is independent of the variable $z$ and multilinear in
$x_{1},\ldots,x_{n}$, is by taking the homogenous component of
degree $k$ in $z$ in the polynomial
$f((1+z)x_{1},\ldots,(1+z)x_{n},\Lambda)$.

\end{remark}

\begin{lemma}
\label{delta_commutative}Notation as above. Suppose that $z=p(\Sigma)$
is independent of the variables $x_{1},\ldots,x_{n}$.  The
following hold.
\begin{enumerate}
\item If $f$ is multilinear and alternating on $x_{1},\ldots,x_{n}$ then
so does $\delta_{k}^{z}(f)$.
\item
$\delta_{k}^{z}$ and $\delta_{k'}^{z}$ commute.
\end{enumerate}
\end{lemma}

\begin{proof}
\textbf{(1)}~Suppose $k>1$ and $1\leq i\neq j\leq n$. Write:
\begin{multline}
\delta_{k}^{z}(f)=\sum_{S\in A_{k}}f|_{\forall t\in S:\, zx_{t}\rightarrow x_{t}}+\sum_{S\in A_{k-1}}f|_{zx_{i}\rightarrow x_{i},\,\,\forall t\in S:\, zx_{t}\rightarrow x_{t}}+\\
+\sum_{S\in A_{k-1}}f|_{zx_{j}\rightarrow x_{j},\,\forall t\in S:\, zx_{t}\rightarrow x_{t}}+\sum_{S\in A_{k-2}}f|_{zx_{i}\rightarrow x_{i},zx_{j}\rightarrow x_{j},\,\,\forall t\in S:\, zx_{t}\rightarrow x_{t}}\label{eq:1-1}
\end{multline}
where $A_{k}$ contains all the subsets of $\{1,\ldots,n\}$ of size
$k$ not containing $i$ nor $j$. We need to show that
$\delta_{k}^{z}(f)|_{x_{i}=x_{j}}=0$. We do so by showing that the
first sum, the last sum and the sum of the two middle sums are zero
separately:
\begin{eqnarray*}
\sum_{S\in A_{k}}\left(f|_{\forall t\in S:\, zx_{t}\rightarrow x_{t}}\right)|_{x_{i}=x_{j}} & = & \left(\sum_{S\in A_{k}}\left(f|_{x_{i}=x_{j}}\right)|_{\forall t\in S:\, zx_{t}\rightarrow x_{t}}\right)|_{x_{i}=x_{j}}=0\\
\sum_{S\in A_{k-2}}\left(f|_{zx_{i}\rightarrow x_{i},zx_{j}\rightarrow x_{j},\,\,\forall t\in S:\, zx_{t}\rightarrow x_{t}}\right)|_{x_{i}=x_{j}} & = & \left(\sum_{S\in A_{k-2}}\left(f|_{zx_{i}\rightarrow x_{i},zx_{i}\rightarrow x_{j}}\right)|_{\forall t\in S:\, zx_{t}\rightarrow x_{t}}\right)|_{x_{i}=x_{j}}=0
\end{eqnarray*}
and
\begin{multline*}
\left(\sum_{S\in A_{k-1}}\left(f|_{zx_{i}\rightarrow x_{i},\,\,\forall t\in S:\, zx_{t}\rightarrow x_{t}}\right)+\sum_{S\in A_{k-1}}\left(f|_{zx_{j}\rightarrow x_{j}\,\,\forall t\in S:\, zx_{t}\rightarrow x_{t}}\right)\right)|_{x_{i}\rightarrow x_{j}}=\\
\sum_{S\in A_{k-1}}f|_{zx_{i}\rightarrow x_{i},x_{i}\rightarrow x_{j}\,\,\forall t\in S:\, zx_{t}\rightarrow x_{t}}-\left(\sum_{S\in A_{k-1}}\left(f|_{x_{i}\leftrightarrow x_{j}}\right)|_{zx_{j}\rightarrow x_{j},\,\,\forall t\in S:\, zx_{t}\rightarrow x_{t}}\right)|_{x_{i}\rightarrow x_{j}}=\\
\sum_{S\in A_{k-1}}f|_{zx_{i}\rightarrow x_{i},x_{i}\rightarrow x_{j}\,\,\forall t\in S:\, zx_{t}\rightarrow x_{t}}-\left(\sum_{S\in A_{k-1}}f|_{x_{i}\rightarrow x_{j},zx_{j}\rightarrow x_{i},\,\,\forall t\in S:\, zx_{t}\rightarrow x_{t}}\right)|_{x_{i}\rightarrow x_{j}}=\\
\sum_{S\in A_{k-1}}f|_{zx_{i}\rightarrow x_{i},x_{i}\rightarrow
x_{j}\,\,\forall t\in S:\, zx_{t}\rightarrow
x_{t}}-\left(\sum_{S\in A_{k-1}}f|_{zx_{i}\rightarrow x_{i},
x_{i}\rightarrow x_{j}, \,\,\forall t\in S:\, zx_{t}\rightarrow
x_{t}}\right)=0
\end{multline*}
All in all we get $0$. We still need to consider the cases $k=0$
and $k=1$. The case $k=0$ is trivial since $\delta_{0}^{z}(f)=f$.
In case $k=1$ we get the same proof as for $k>1$ with the only difference that the
last sum in $($\ref{eq:1-1}$)$ does not appear here.

\textbf{(2)~}Denote by $L_{k}$ the set containing all the subsets
of $\{1,\ldots,n\}$ of size $k$. Then
\[
\delta_{k}^{z}(f)=\sum_{S\in L_{k}}f|_{\forall i\in S:\, zx_{i}\rightarrow x_{i}}
\]
and
\[
\delta_{k'}^{z}(\delta_{k}^{z}(f))=\sum_{S_{2}\in
L_{k'}}\sum_{S_{1}\in L_{k}}\left(f|_{\forall i\in S_{1}:\,
zx_{i}\rightarrow x_{i}}\right)|_{\forall j\in S_{2}:\,
zx_{j}\rightarrow x_{j}}=\sum_{S_{1}\in L_{k},S_{2}\in
L_{k'}}f|_{\forall i\in S_{1}\cap S_{2}:\, z^{2}x_{i}\rightarrow
x_{i},\,\,\forall i\in S_{1}\oplus S_{2}:\, zx_{i}\rightarrow
x_{i}}
\]
where $S_{1}\oplus S_{2} = (S_{1} \setminus S_{2}) \cup (S_{2}
\setminus S_{1})$. This proves the claim. Notice that the second
equality holds because $z$ is independent of
$\{x_{1},\ldots,x_{n}\}$.
\end{proof}

Suppose $W$ is an affine algebra with Kemer index $(n,r)$. We
let $\mu$ be the minimal number of small sets in its Kemer
polynomials.  Let $f=f(X,\Lambda)$ be a Kemer polynomial of $W$
($X$ and $\Lambda$ are sets of variables) and suppose that
(already) $\mu$ small sets and all big sets of variables are
contained in $\Lambda$.

\begin{lemma}
\label{lem:technical}Suppose $f(x_{1},\ldots,x_{n+1},\Lambda)$ is such a Kemer polynomial which in addition
is multilinear on $x_{1},\dots,x_{n+1}$ and alternates on
$x_{1},\ldots,x_{n}$. Then
\[
\sum_{t=0}^{n}(-1)^{t}\delta_{t}^{z}\left(f|_{z^{n-t}x_{n+1}\rightarrow
x_{n+1}}\right)\in \Id(W)
\]
\end{lemma}
\begin{proof}
It easy to check that if $g(x_{1},\ldots,x_{n+1},E)$ , $E$ any set of variables, is
multilinear on $x_{1},\ldots,x_{n+1}$ and alternates on
$x_{1},\ldots,x_{n}$, then
\[
\tilde{g}=g(x_{1},\ldots,x_{n+1},E)-\sum_{k=1}^{n}g(x_{1},\ldots,x_{k-1},x_{n+1},x_{k+1},..,x_{n},x_{k},E)
\]
alternates on $x_{1},\ldots,x_{n+1}$. Therefore, if we replace
$g$ by $\delta_{t}^{z}(f)$ we obtain that
\[
\widetilde{\delta_{t}^{z}(f)}=\left(\delta_{t}^{z}(f)\right)-\sum_{k=1}^{n}\left(\delta_{t}^{z}(f)\right)|_{x_{k}\leftrightarrow
x_{n+1}}
\]
has $\mu$ alternating small sets and $r+1$ alternating big sets.
Thus, $\widetilde{\delta_{t}^{z}(f)}$ is an identity of $W$.

Substituting $z^{n-t}x_{n+1}$ in $x_{n+1}$ yields:
\begin{eqnarray*}
\left(\delta_{t}^{z}(f)\right)|_{z^{n-t}x_{n+1}\rightarrow x_{n+1}} & \equiv & \sum_{k=1}^{n}\left(\left(\delta_{t}^{z}(f)\right)|_{x_{k}\leftrightarrow x_{n+1}}\right)
|_{z^{n-t}x_{n+1}\rightarrow x_{n+1}}\,\,\mbox{mod}\, (\Id(W))\\
 & = & \sum_{k=1}^{n}\left(\delta_{t}^{z}(f)\right)|_{x_{k}\rightarrow x_{n+1},z^{n-t}x_{n+1}\rightarrow x_{k}}
\end{eqnarray*}
Since the operator $\delta_{t}^{z}$ ``has no effect'' on
$x_{n+1}$, one can see easily that the operation of
$\delta_{t}^{z}$ commutes with the substitution
${z^{n-t}x_{n+1}\rightarrow x_{n+1}}$, that is
$\left(\delta_{t}^{z}(f)\right)|_{z^{n-t}x_{n+1}\rightarrow
x_{n+1}}=\delta_{t}^{z}\left(f|_{z^{n-t}x_{n+1}\rightarrow
x_{n+1}}\right)$ and so we have
\[
\delta_{t}^{z}\left(f|_{z^{n-t}x_{n+1}\rightarrow
x_{n+1}}\right)\equiv\sum_{k=1}^{n}\left(\delta_{t}^{z}(f)\right)|_{x_{k}\rightarrow
x_{n+1},z^{n-t}x_{n+1}\rightarrow x_{k}}\,\,\mbox{mod}\,
(\Id(W)).
\]
Therefore, the lemma will be proved if we show
\[
\sum_{t=0}^{n}\sum_{k=1}^{n}(-1)^{t}\left(\delta_{t}^{z}(f)\right)|_{x_{k}\rightarrow
x_{n+1},z^{n-t}x_{n+1}\rightarrow x_{k}}\in \Id(W)
\]
and after changing the order of summation, it is clear that
it is enough to show that for every $k$
\[
\sum_{t=0}^{n}(-1)^{t}\left(\delta_{t}^{z}(f)\right)|_{x_{k}\rightarrow
x_{n+1},z^{n-t}x_{n+1}\rightarrow x_{k}}=0.
\]
For simplicity (and in fact, without loss of generality) we show the statement in case $k=1$. The following equality
holds
\begin{eqnarray*}
\left(\delta_{t}^{z}(f)\right)|_{x_{1}\rightarrow x_{n+1},z^{n-t}x_{n+1}\rightarrow x_{1}} & = & \underset{g_{t}}{\underbrace{\sum_{1<i_{1}<\cdots<i_{t}\leq n}
\left(f_{i_{1},\ldots,i_{t}}\right)|_{x_{1}\rightarrow x_{n+1},z^{n-t}x_{n+1}\rightarrow x_{1}}}}\\
+\underset{h_{t}}{\underbrace{\sum_{1=i_{1}<i_{2}<\cdots<i_{t}\leq
n}\left(f_{i_{1},\ldots,i_{t}}\right)|_{x_{1}\rightarrow
x_{n+1},z^{n-t}x_{n+1}\rightarrow x_{1}}}}
\end{eqnarray*}
 Thus,
\begin{eqnarray*}
g_{t} & = & \sum_{1<i_{1}<\cdots<i_{t}\leq n}f|{}_{x_{1}\leftarrow z^{n-t}x_{n+1},x_{i_{1}}\leftarrow zx_{i_{1}},\ldots,x_{i_{t}}\leftarrow zx_{i_{t}},x_{n+1}\leftarrow x_{1}}\\
h_{t} & = & \sum_{1<i_{2}<\cdots<i_{t}\leq n}f|_{x_{1}\leftarrow
z^{n-t+1}x_{n+1},x_{i_{2}}\leftarrow
zx_{i_{2}},\ldots,x_{i_{t}}\leftarrow zx_{i_{t}},x_{n+1}\leftarrow
x_{1}}.
\end{eqnarray*}
Observe that $g_{n}=h_{0}=0$ and $h_{t+1}=g_{t}$ for $t=0,\ldots,
n-1$ and so
\[
\sum_{t=0}^{n}(-1)^{t}\left(\delta_{t}^{z}(f)\right)|_{x_{1}\rightarrow
x_{n+1},z^{n-t}x_{n+1}\rightarrow
x_{1}}=\sum_{t=0}^{n}(-1)^{t}(h_{t}+g_{t})=0.
\]

\end{proof}
We will use the above lemma in the following setting.
\begin{corollary} \label{cor:Cayeley-Hamilton}

Let
$f(x_{1},\ldots,x_{n+1},\Lambda)=\sum_{\sigma}f_{1,\sigma}x_{n+1}f_{2,\sigma}$
be a Kemer polynomial which is multilinear on $x_{1},\ldots,x_{n+1}$ and
alternating on  $x_{1},\ldots,x_{n}$. Suppose the variable $z$ is
not one of the variables $x_{1},\ldots,x_{n}$. Then
\[
\sum_{t=0}^{n}(-1)^{t}\delta_{t}^{z}\left(\sum_{\sigma}f_{1,\sigma}z^{n-t+1}f_{2,\sigma}\right)\in \Id(W).
\]
\end{corollary}
\begin{proof}
Use the previous lemma to obtain:
\[
\sum_{t=0}^{n}(-1)^{t}\delta_{t}^{z}\left(\sum_{\sigma}f_{1,\sigma}z^{n-t}x_{n+1}f_{2,\sigma}\right)\in \Id(W).
\]
Now substitute $x_{n+1}\leftarrow z$.
\end{proof}

We will be interested in the following special case.
\begin{corollary}\label{final expression for the use of interpretation}

Let
$f(x_{1},\ldots,x_{n+1},\Lambda)=\sum_{\sigma}f_{1,\sigma}x_{n+1}f_{2,\sigma}$ as in the corollary above. Let $F\langle x_{1},\ldots,x_{n},\Sigma \rangle$ be an affine
free algebra where $\Sigma$ is any set of variables
disjoint to the set $\{x_{1},\ldots,x_{n}\}$ and let $p=p(\Sigma) \in F\langle
x_{1},\ldots,x_{n},\Sigma \rangle$. Then

 \[
\sum_{t=0}^{n}(-1)^{t}\delta_{t}^{p}\left(\sum_{\sigma}f_{1,\sigma}p^{n-t+1}f_{2,\sigma}\right)\in \Id(W) \cap F\langle x_{1},\ldots,x_{n},\Sigma \rangle.
\]

\end{corollary}
\end{subsection}

\subsection{Interpretation Lemma}
\begin{lemma}
\label{lem:interpretation}Let $A$ be an $F$-algebra and let $L$ be
an ideal of $A$. Suppose the polynomial ring
$R=F[t_{1},\ldots,t_{n}]$ acts $F-$linearly on $L$ such that
$f\cdot(ax)=\left(f\cdot a\right)x$ and $f\cdot(xa)=x\left(f\cdot
a\right)$, for all $x\in A$, $a\in L$ and $f\in R$ $($that is the
action of $R$ on $L$ commutes with the $A$-bimodule structure on
$L$$)$. Then the natural map $A\rightarrow
A'=R\otimes_{F}A/(t_{i}\otimes a-1\otimes(t_{i}\cdot a))|a\in L,\,
i=1,\ldots,n)$ is an embedding.
\end{lemma}
\begin{proof}

Consider the exact sequence $0 \rightarrow \ker \nu \rightarrow R
\otimes_{F}L \rightarrow R \otimes_{R} L\cong L \rightarrow 0$
where $\nu: R \otimes_{F}L \rightarrow R \otimes_{R} L$ is the
natural map of $F$-algebras. Clearly $\ker \nu = \langle
t_{i}\otimes a-1\otimes(t_{i}\cdot a)|a\in L,\,
i=1,\ldots,n\rangle_{R \otimes_{F}L}$. However, the ideal $\ker \nu
\subseteq R \otimes_{F}L$ is invariant under the action of $R
\otimes_{F}A$ and hence if we denote by $V$ the vector space that
supplements $L$ in $A$, we have

$$A'=R\otimes_{F}A/\ker \nu \cong R\otimes_{F} (L \oplus V)/\ker \nu
$$
$$
= (R\otimes_{F} L \oplus R\otimes_{F} V)/\ker \nu \cong
(R\otimes_{F} L/\ker \nu) \oplus R\otimes_{F} V \cong (L \oplus
R\otimes_{F} V ).$$ This shows the natural  map $A \rightarrow A'$
is an embedding and the lemma is proved.
\end{proof}

\section{Representable spaces\label{Representable-Spaces}}

Let us summarize what we have and what remains to be done. We are
assuming that $\Gamma$ is the $T$-ideal of an affine algebra $W$
with Kemer index $p=(n,r)$ and $S_{p}$ is the $T$-ideal generated
by its Kemer polynomials (with at least $\mu$ small sets). The
idea is to proceed by induction on the Kemer index of $\Gamma$. We
assume the main theorem holds for all affine algebras with Kemer
index smaller than $p$ and prove it for $\Gamma$ (note that
$p=(0,0)$ if and only if $W=0$). Consider the $T$-ideal
$\Gamma'=\Gamma + S_{p}$. Clearly, its Kemer index $p'$ is smaller
than $p$ (for otherwise, any Kemer polynomial of $\Gamma'$ is a
Kemer polynomial of $\Gamma$ and hence in $S_{p}$) and hence there
is a representable algebra $A'$ (or a finite dimensional algebra
over a field extension) with $\Id(A')=\Gamma'$.

The ingredient we are still missing is
the existence of a representable algebra $B_{p}$ satisfying all
the identities of $W$ and such that any polynomial in $S_{p}$ (which is
not in $\Gamma$) is a nonidentity of $B_{p}$. Then it will be easy to conclude
that the representable algebra $A'\times B_{p}$ is PI equivalent to $W$.

Let $\mathcal{W}_{0}=F\langle \Sigma \rangle/\Id_{F\langle \Sigma
\rangle}(W)$ be an affine relatively free algebra of $W$. Here,
$\Id_{R}(W)$ denotes the ideal of an algebra $R$ generated by all
evaluations of $\Id(W)$ on $R$. Denote by $X_{p}$ a set of
$n(\mu+1)+(n+1)r$ variables (i.e. precisely the number of
variables needed to support the $\mu +1$ small sets and $r$ big
sets in a Kemer polynomial). Suppose the set $X_{p}$ is disjoint
to the variables $\Sigma$ (that generate $\mathcal{W}_{0}$) and
denote by $\mathcal{W}=F\langle X_{p},\Sigma \rangle/\Id_{F\langle
X_{p},\Sigma \rangle}(W)$ the affine relatively free algebra of
$W$ generated by $\Sigma$ and $X_{p}$.

Let us fix (for the rest of this section) a decomposition of
$X_{p}$ into $\mu+1$ sets of variables
$X_{1},\ldots,X_{\mu+1}$, each containing exactly $n$ elements, and
$r$ additional sets $X_{\mu+2},\ldots,X_{\mu+1+r}$, each containing
$n+1$ variables.

Let $f$ be a Kemer polynomial of $W$ with at least $\mu+1$ small sets.

Before getting into the definitions and the precise construction
of the representable algebra $B_{p}$, let us give here a short
outline of the construction. We consider the affine relatively
free algebra $\mathcal{W}=F\langle X_{p},\Sigma
\rangle/\Id_{F\langle X_{p},\Sigma \rangle}(W)$ of $W$ generated by
$\Sigma$ and $X_{p}$. The set $X_{p}$ was already ``fragmented''
into $\mu+1$ sets of variables $X_{1},\ldots,X_{\mu+1}$, each
containing exactly $n$ elements, and $r$ additional sets
$X_{\mu+2},\ldots,X_{\mu+1+r}$, each containing $n+1$ variables,
in particular, sufficiently many small sets and big sets to
support nonvanishing evaluations of any Kemer polynomial $f$ on
$\mathcal{W}=F\langle X_{p},\Sigma \rangle/\Id_{F\langle
X_{p},\Sigma \rangle}(W)$ in such a way that precisely $\mu+1$
small sets of such $f$ take values precisely in the $\mu+1$ small
sets $X_{1},\ldots,X_{\mu+1}$ (modulo $\Id_{F\langle X_{p},\Sigma
\rangle}(W)$) and all big sets of $f$ are evaluated (in one to one
correspondence) on the sets $X_{\mu+2},\ldots,X_{\mu+1+r}$. It is
not difficult to show that any Kemer polynomial $f$ has a nonzero
evaluation of that kind. These are by definition, the
\textit{admissible evaluations}. All such (nonzero evaluations)
evaluations span a vector space $V$ in $\mathcal{W}$. Our goal is
to mod out ideals of $\mathcal{W}$ such that at the end we obtain
a representable algebra $B_{p}$, and yet the space $V$ embeds in
$B_{p}$. This will show that the Kemer polynomials are
nonidentities of $B_{p}$. What are the ideals we mod out by? We
consider a Shirshov base in $\mathcal{W}=F\langle X_{p},\Sigma
\rangle/\Id_{F\langle X_{p},\Sigma \rangle}(W)$ represented by
monomials $z$ which either contain elements of $X_{p}$ or not. If
$z$ is such a monomial (i.e. $z+ \Id_{F\langle X_{p},\Sigma
\rangle}(W)$ is an element of the Shirshov base) that contains a
variable of $X_{p}$, then by modding out with the ideal $I$ of
$\mathcal{W}$ generated (or in fact consisting) of elements of
$F\langle X_{p},\Sigma \rangle$ in which at least one element of
$X_{p}$ appears twice, then the element $z+ \Id_{F\langle
X_{p},\Sigma \rangle}(W)$ is nilpotent modulo $I$ and hence
integral. On the other hand it is easy to see that the space $V$
intersects trivially the ideal $I$. Most of the efforts are
devoted to construct ideals such that by modding out them
successively, each element of the Shirshov base which is free of
elements of $X_{p}$, one at a time, becomes integral and yet the
space $V$ embeds.

\begin{definition}

An evaluation of $f$ on $F\langle X_{p},\Sigma \rangle$ is admissible if the following conditions are satisfied.

\begin{enumerate}

\item

Precisely $\mu+1$ small sets of $f$, say $\dot{X}_{1},\ldots,\dot{X}_{\mu+1}$, are evaluated bijectively on the sets $X_{1},\ldots,X_{\mu+1}$

\item

All big sets of $f$ are evaluated bijectively on the sets $X_{\mu+2},\ldots,X_{\mu+1+r}$

\item

The rest of the variables of $f$ are evaluated on $F\langle \Sigma \rangle$
\end{enumerate}

An evaluation of $f$ on $\mathcal{W}$ is admissible if it is represented by an admissible evaluation on $F\langle X_{p},\Sigma \rangle$.

\end{definition}

We denote by $\mathcal{S}_{p}$ the $F$-span (in $\mathcal{W}$) of all admissible evaluations of all Kemer polynomials of $W$.

Our goal in this section is to prove that $\mathcal{S}_{p}$ is a
representable space of $\mathcal{W}$. Here is the precise
definition.
\begin{definition}
Let $W$ be a PI $F$-algebra and let $S$ be an $F$-subspace of $W$. We say that
$S$ is an representable space of $W$ if there exist a
representable algebra $B$ with $\Id(B) \supseteq \Id(W)$ and a homomorphism
\[
\phi:W\rightarrow B
\]
such that $\phi$ maps $S$ isomorphically into $B$.
\end{definition}

\begin{remark}

Our main difficulty in the construction of $B$ is that on one
hand it should be not ``too big'' so that it is representable and
on the other hand not ``too small'' so that $\mathcal{S}_{p}$
embeds.

The compromise is achieved by ``forcing'' a Shirshov base of $\mathcal{W}$ to be integral over
some commutative Noetherian $F$-algebra.
\end{remark}

Let $Y$ be a Shirshov base of $\mathcal{W}$ consisting of elements
which are represented by monomials on the set $\Sigma \sqcup
X_{p}$. Denote by $Y_{0}=\{b_{1}+\Id_{F\langle X_{p},\Sigma
\rangle}(W),\ldots, b_{t}+\Id_{F\langle X_{p},\Sigma \rangle}(W)\}$
the elements of the Shirshov base where the representing
monomials are independent of $X_{p}$ and let $Y_{1}= Y\setminus
Y_{0}$ be the remaining elements of the Shirshov base.

Consider the algebra $\mathcal{U}=\mathcal{W}/I$ where $I$ is
generated by elements of the form $xwx+\Id_{F\langle X_{p},\Sigma
\rangle}(W)$ and $x^{2}+\Id_{F\langle X_{p},\Sigma \rangle}(W)$
where $x\in X_{p}$ and $w\in F\langle X_{p},\Sigma \rangle$.
Denote by $\phi: \mathcal{W} \rightarrow \mathcal{U}$ the natural
map. It is clear from the definition of $\mathcal{S}_{p}$ that
$\mathcal{S}_{p}\cap I=0$ and so, by abuse of notation, we write
$\mathcal{S}_{p}$ also for the (isomorphic) image
$\phi(\mathcal{S}_{p})$ in $\mathcal{U}$. Note that the elements
of $\overline{Y}_{1}=\phi(Y_{1})$ are nilpotent and hence
integral over $F$. Our goal is then to ``force'' $\phi(Y_{0})$, the
remaining elements of the Shirshov base, to become integral over
a suitable Noetherian ring, yielding a representable algebra and yet an algebra where the space $\mathcal{S}_{p}$ embeds.

\begin{remark}
As defined above, the elements $b_{i}$, $i=1,\ldots,t$, are
monomials in the affine free algebra $F\langle X_{p}, \Sigma
\rangle$. In the sequel, these elements will be considered in
different quotients of $F\langle X_{p}, \Sigma \rangle$. By abuse
of notations we denote them by $\bar{b}_{i}$, $i=1,\ldots,t$. This
should not confuse the reader.

\end{remark}

Let $\overline{Y}_{0}=\phi(Y_{0})=\{\bar{b}_{1},\ldots,\bar{b}_{t}\}$ and define by
induction for $i=1,\ldots,t$
\[
B_{p}^{(i)}=\frac{F\left[\theta_{1}^{(i)},\ldots,\theta_{n}^{(i)}\right]\otimes
B_{p}^{(i-1)}}{J_{i}},\,\, B_{p}^{(0)}=\mathcal{U}
\]
where $F\left[\theta_{1}^{(i)},\ldots,\theta_{n}^{(i)}\right]$ is a polynomial $F$-algebra and $J_{i}$ is the ideal generated by the
element
\[
1\otimes \bar{b}_{i}^{n+1}+\sum_{k=1}^{n}(-1)^{k}\theta_{k}^{(i)}\otimes
\bar{b}_{i}^{n-k+1}.
\]

(We remind the reader that $n$ is the size of a small set in Kemer polynomials)

In other words,
\[
B_{p}^{(i)}=\frac{R_{i}\otimes F\left\langle
X_{p},\Sigma\right\rangle }{J^{(i)}+I+\Id_{R_{i}\left\langle
X_{p},\Sigma\right\rangle}(W)}
\]
where $R_{i}=F\left[\theta_{1}^{(1)},\ldots,\theta_{n}^{(1)},\ldots,\theta_{1}^{(i)},\ldots,\theta_{n}^{(i)}\right]$
and $J^{(i)}$ is generated by
\[
1\otimes b_{j}^{n+1}+\sum_{k=1}^{n}(-1)^{k}\theta_{k}^{(j)}\otimes b_{j}^{n-k+1},\, j=1,\ldots,i.
\]

Denote by $\mathcal{S}_{p}^{(i)}$ the ideal generated by the
projection of $\mathcal{S}_{p}$ into $B_{p}^{(i)}$, and by
$\mathcal{J}^{(i)}$ the ideal $J^{(i)}+I+\Id_{R_{i}\left\langle
X_{p},\Sigma\right\rangle}(W)$.

\begin{lemma}\label{well definition of operators}
Let $X$ be any subset of $X_{p}$ of cardinality $m$ and $z$ be any
monomial consisting of variables of $\Sigma$. Then the following
actions on $R_{i}\left\langle X_{p},\Sigma\right\rangle $ preserve
the ideal $\mathcal{J}^{(i)}$ (note that here we do not insist in the decomposition of $X_{p}$ into the subsets $X_{i}$ we fixed above).
\begin{enumerate}
\item The Zubrilin-Razmyslov traces $\delta_{k}^{z}=\delta_{k}^{z}|_{X}$, where $k=0,\ldots,m$.
\item The alternation of the set $X$ which is defined by
\[
Alt_{X}(f)=\sum_{\sigma\in S_{X}}(-1)^{\sigma}f|_{\forall x\in X:\,\sigma(x)\rightarrow x}.
\]

\item Zero substitution on elements of $X$. That is
\[
F_{X}(f)=f|_{\forall x\in X:\,0\rightarrow x}.
\]

\end{enumerate}
\end{lemma}
\begin{proof}
Since the actions above are linear, it is sufficient to check that
they preserve the ideals $\Id_{R_{i}\left\langle
X_{p},\Sigma\right\rangle}(W),I$ and $J^{(i)}$.

Since $\Id_{R_{i}\left\langle X_{p},\Sigma\right\rangle}(W)$ is a
$T$-ideal it is clearly preserved by each of the actions defined
in $(1)-(3)$. Next we turn to the ideal $J^{(i)}$, starting with (3).
Suppose $g,h\in R_{i}\left\langle X_{p},\Sigma\right\rangle $ are
monomials. Then for any $j=1,\dots,i$ and any nonnegative integer
$k$, $F_{X}\left(gb_{j}^{k}h\right)=0$ if some variable of $X$ is
in $gh$ whereas $F_{X}\left(gb_{j}^{k}h\right)=gb_{j}^{k}h$ if no
variable of $X_{p}$ is in $gh$. Hence
\[
F_{X}\left(1\otimes g\left(1\otimes b_{j}^{n+1}+\sum_{k=1}^{n}(-1)^{k}\theta_{k}^{(j)}\otimes b_{j}^{n-k+1}\right)1\otimes h\right)
\]
is zero or
$1\otimes g\left(1\otimes b_{j}^{n+1}+\sum_{k=1}^{n}(-1)^{k}\theta_{k}^{(j)}\otimes b_{j}^{n-k+1}\right)1\otimes h$
which are elements of $J^{(i)}$.

Let us show $Alt|_{X}$ preserves $J^{(i)}$. Since the element
$1\otimes b_{j}^{n+1}+\sum_{k=1}^{n}(-1)^{k}\theta_{k}^{(j)}\otimes b_{j}^{n-k+1}$
does not contain any variable of $X$, we have for $j=1,\ldots,i$
\[
Alt|_{X}\left(1\otimes g\left(1\otimes b_{j}^{n+1}+\sum_{k=1}^{n}(-1)^{k}\theta_{k}^{(j)}\otimes b_{j}^{n-k+1}\right)1\otimes h\right)
\]
is a sum of polynomials of the type
\[
1\otimes g_{\sigma}\left(1\otimes b_{j}^{n+1}+\sum_{k=1}^{n}(-1)^{k}\theta_{k}^{(j)}\otimes b_{j}^{n-k+1}\right)1\otimes h_{\sigma}\in J^{(i)}.
\]
Finally, a similar reasoning shows that $\delta_{k}^{z}$ also preserves
$J^{(i)}$.

We now turn to show the invariance of the ideal $I$. Consider
elements $1\otimes g,1\otimes h,1\otimes w\in R_{i}\otimes
F\left\langle X_{p},\Sigma\right\rangle$ where $g,h,w$ are
monomials. Let $x_{0} \in X$. As we mentioned above the action of
$F_{X}$ on $1\otimes g$ is either zero or $1\otimes g$. Hence,
since $I$ is generated by elements in $\mathcal{W}$ represented by
monomials in $R_{i}\langle X_{p}, \Sigma \rangle$, it is evident
that $I$ is preserved by $F_{X}$. The result of acting with
$Alt_{X}$ or $\delta_{k}^{z}$ on an element of the form $1\otimes
gx_{0}^{2}h$ or $1\otimes gx_{0}wx_{0}h$, is a sum of monomials each
having some element of $X_{p}$ appearing at least twice.
Therefore, $I$ is preserved by $\delta_{k}^{z}$ and $Alt_{X}$.
\end{proof}
For future reference we record the conclusion here.

\begin{corollary}{\label{Well definition of actions}}

The operations considered above $($on $R_{i}\otimes F\langle X_{p}, \Sigma \rangle$$)$, namely $\delta_{k,
X_{l}}^{b_{i}}, Alt_{X_{l}}$, $l=1,\ldots,\mu+r+1$ and $F_{X}$,
where $X=\{x\}$ and $x \in X_{p}$, determine well defined actions on $B_{p}^{(i)}$,
$($i=1,\ldots,t$)$.

\end{corollary}
\begin{lemma}\label{The-following-holds(1)}The following holds:
\begin{enumerate}
\item $B_{p}=B_{p}^{(t)}$ is representable.
\item $\mathcal{S}_{p}$ is mapped isomorphically into $B_{p}^{(i)}$ for
$i=1,\ldots,t$.
\end{enumerate}
\end{lemma}
\begin{proof}
By construction, the algebra $B_{p}$ has an integral Shirshov base
over the polynomial algebra $R_{t}$ and hence is representable by
Lemma \ref{finite-module} and Theorem \ref{representable-over-commutative}.
This proves the first part of the lemma. The second part however requires
more work.

We proceed by induction on $i$. Since $\mathcal{S}_{p}$ embeds in
$\mathcal{U}$ we have the result for $i=0$. Let $i>0$ and suppose
$(2)$ holds for $i-1$, namely $\mathcal{S}_{p}$ embeds in $B_{p}^{(i-1)}$.
We denote the image in $B_{p}^{(i-1)}$ again by $\mathcal{S}_{p}$.

Let $f=f(X_{p},\Sigma)$ be an admissible evaluation of a Kemer
polynomial of $W$. Suppose $\overline{f}\in J_{i}$. We need to
show $\overline{f}$ is zero in $B_{p}^{(i-1)}$.

By the definition of $J_{i}$, $1\otimes \overline{f}$ has the form

\begin{equation}
1\otimes \overline{f}=\sum_{j}v_{j}\otimes
\bar{g}_{j}\cdot\underset{\bar{B}}{\underbrace{\left(1\otimes
\bar{b}^{n+1}+\sum_{k=1}^{n}(-1)^{k}\theta_{k}\otimes\bar{b}^{n-k+1}\right)}}1\otimes\bar{h}_{j}\label{equate1}
\end{equation}
where $g_{j},h_{j}\in R_{i-1}\otimes F\left\langle
X_{p},\Sigma\right\rangle $ are monomials in the variables of
$\Sigma \sqcup X_{p}$ (with coefficients in $R_{i-1}$) and
$v_{j}\in F\left[\theta_{1},\ldots,\theta_{n}\right]$. Here,
$b=b_{i},\theta_{1}=\theta_{1}^{(i)},\ldots,\theta_{n}=\theta_{n}^{(i)}$
and $B= \left(1\otimes
b^{n+1}+\sum_{k=1}^{n}(-1)^{k}\theta_{k}\otimes b^{n-k+1}\right)$.
Note that with this notation the equality in (2) takes place in
$F\left[\theta_{1},\ldots,\theta_{n}\right]\otimes B_{p}^{(i-1)}$.

We claim that we may assume that for every $j$ in the summation
above the elements $g_{j}$ and $h_{j}$ are such that all variables
of $X_{p}$ appear exactly once in either $g_{j}$ or $h_{j}$ (but
not in both). Indeed, suppose $A_{x}$ is the set of all indexes
$j$ such that $g_{j}h_{j}$ contains the variable $x\in X$.
Applying $F_{\{x\}}$ to $\bar{f}-\sum_{j\in
A_{x}}(v_{j}\otimes\bar{g}_{j})\cdot\bar{B}\cdot
(1\otimes\bar{h}_{j})$ yields zero whereas the action of
$F_{\{x\}}$ on the remaining terms is trivial and so can be
ignored in the sum above. Repeating this process to all the
variables of $X_{p}$ we obtain that (we may assume) each variable
of $X_{p}$ appears in every $g_{j}h_{j}$. The claim is now proved
since modding by $I$ garanties that if there is a variable of
$X_{p}$ that appears more than once in some $(1\otimes g_{j})\cdot
B\cdot (1\otimes h_{j})$, then
$(1\otimes\bar{g}_{j})\bar{B}(1\otimes \bar{h}_{j})=0$.

Now we wish to apply $Alt_{X}$ for $X=X_{1},\ldots,X_{\mu+r+1}$ on
the two sides of (\ref{equate1}). To this end consider for each
$j$ the polynomial

$$M_{j}=(v_{j}\otimes g_{j})B(1\otimes h_{j})\in F[\theta_{1},\ldots,\theta_{n}]\otimes (R_{i-1}\otimes F\left\langle X_{p},\Sigma\right\rangle)$$
representing

$$\overline{M}_{j}=(v_{j}\otimes \bar{g}_{j})\bar{B}(1\otimes \bar{h}_{j})\in F[\theta_{1},\ldots,\theta_{n}]\otimes B_{p}^{(i-1)}.$$

Alternating all variables of $X_{p}$ (with respect to the fixed decomposition into small and big sets) in $M_{j}$ we obtain

$$
L_{j}=\sum_{\sigma}M_{j,\sigma}
$$
and clearly $\overline{L}_{j}=\sum_{\sigma}\overline{M}_{j,\sigma}
\in F[\theta_{1},\ldots,\theta_{n}]\otimes B_{p}^{(i-1)}$ (to
simplify notation we include the sign arising from alternation in
$M_{j,\sigma}$).

Now, by Corollary \ref{Well definition of actions} we know that alternation of elements of $X_{p}$ is well defined on $B_{p}^{(i-1)}$, so we obtain

$$
\overline{L}_{j}=\sum_{\sigma}\overline{M}_{j,\sigma}=1\otimes \alpha \overline{f}
$$
where $\alpha = (n!)^{\mu + 1}
((n+1)!)^{r}$.

Now, we wish to apply the interpretation lemma (Lemma
\ref{lem:interpretation}) on the algebra
$F[\theta_{1},\ldots,\theta_{n}]\otimes B_{p}^{(i-1)}$ with the
ideal $\mathcal{S}_{p}^{(i-1)} \subseteq B_{p}^{(i-1)}$ (see
notation above). Indeed, we will provide an interpretation of
$\theta_{1},\ldots,\theta_{n}$ via a suitable action on the ideal
$\mathcal{S}_{p}^{(i-1)}$ such that the image of
$\overline{L}_{j}$ (modulo the interpretation) vanishes and hence
also the image of $\sum_{j}\overline{L}_{j}$, whereas the elements
of $B_{p}^{(i-1)}$ are mapped isomorphically. This will imply that
$\bar{f}=0$ in $B_{p}^{(i-1)}$.

Recall once again (Corollary \ref{Well definition of actions}) that the operation
$\delta_{k}^{b}$ is well defined on $B_{p}^{(i-1)}$ ($b$ is free
of variables of $X_{p}$) so we may set
$$\theta_{k}(q)=\delta_{k}^{b}(q), \space q \in B_{p}^{(i-1)}.$$
However as shown in the lemma below, only the restriction of this
action on the ideal $\mathcal{S}_{p}^{(i-1)}$ commutes with the
$B_{p}^{(i-1)}$-bimodule structure.

\begin{lemma} \label{check conditions of the interpretation lemma}

The following hold:
\begin{enumerate}
\item For every $h\in\mathcal{S}_{p}$, any $k$ and any $x,y\in B_{p}^{(i-1)}$

\[
\delta_{k}^{\bar{b}}(xh)=x\delta_{k}^{\bar{b}}(h),\delta_{k}^{\bar{b}}(hy)=\delta_{k}^{\bar{b}}(h)y
\]

\item for any $k,s=1,\ldots,n$ and $h\in\mathcal{S}_{p}$

\[
\delta_{k}^{\bar{b}}\delta_{s}^{\bar{b}}(h)=\delta_{s}^{\bar{b}}\delta_{k}^{\bar{b}}(h)
\]

\end{enumerate}

\end{lemma}

\begin{proof}
Since $b$ is a monomial having its variables in $\Sigma$, Lemma \ref{delta_commutative}
implies that $\delta_{k}^{b}$ and $\delta_{s}^{b}$ commute, so the
same holds for $\delta_{k}^{\bar{b}}$ and $\delta_{s}^{\bar{b}}$.

To complete the proof of the lemma it suffices to show that if
$h\in\mathcal{S}_{p}$ and $g$ is a monomial in $B_{p}^{(i-1)}$
then $\delta_{k}^{\bar{b}}(gh)=g\delta_{k}^{\bar{b}}(h)$ and
$\delta_{k}^{\bar{b}}(hg)=\delta_{k}^{\bar{b}}(h)g$. Indeed, if
$g$ contains an element of $X_{p}$, then $hg=0=gh$ and also
$g\delta_{k}^{\bar{b}}(h)=0=\delta_{k}^{\bar{b}}(h)g$ verifying
the condition. If $g$ is a monomial in $B_{p}^{(i-1)}$ free of
variables of $X_{p}$, then clearly
$\delta_{k}^{\bar{b}}(hg)=\delta_{k}^{\bar{b}}(h)g$ and
$\delta_{k}^{\bar{b}}(gh)=g\delta_{k}^{\bar{b}}(h)$. This
completes the proof of Lemma \ref{check conditions of the
interpretation lemma}.

\end{proof}

We apply the interpretation lemma on the element $\overline{L}_{j}=\sum_{\sigma}\overline{M}_{j,\sigma} \in F[\theta_{1},\ldots,\theta_{n}]\otimes B_{p}^{(i-1)}$.

We obtain

$$
\sum_{\sigma}\overline{M}_{j,\sigma}= \sum_{\sigma}(v_{j}\otimes \bar{g}_{j,\sigma})\bar{B}(1\otimes \bar{h}_{j,\sigma})=
$$

$$
\sum_{\sigma}(v_{j}\otimes \bar{g}_{j,\sigma})\left(1\otimes \bar{b}^{n+1}+\sum_{k=1}^{n}(-1)^{k}\theta_{k}\otimes\bar{b}^{n-k+1}\right)(1\otimes \bar{h}_{j,\sigma})=
$$
putting $\theta_{0}=1$ we have
$$
\sum_{\sigma}(v_{j}\otimes \bar{g}_{j,\sigma})\left(\sum_{k=0}^{n}(-1)^{k}\theta_{k}\otimes\bar{b}^{n-k+1}\right)(1\otimes \bar{h}_{j,\sigma})=
$$

$$
\sum_{k=0}^{n}(-1)^{k}v_{j}\theta_{k} \otimes \sum_{\sigma}\bar{g}_{j,\sigma}\bar{b}^{n-k+1}\bar{h}_{j,\sigma}
$$

Now, since $\sum_{\sigma}\bar{g}_{j,\sigma}\bar{b}^{n-k+1}\bar{h}_{j,\sigma}\in \mathcal{S}_{p}^{(i-1)}$, we may apply the interpretation

$$\theta_{k}\otimes q = 1 \otimes \delta_{k}^{\bar{b}}(q), k=0,\ldots,n$$ (note that $\delta_{0}$ is the identity map)
and obtain

$$
v_{j} \otimes \sum_{k=0}^{n}(-1)^{k}\delta_{k}^{\bar{b}}\left(\sum_{\sigma}\bar{g}_{j,\sigma}\bar{b}^{n-k+1}\bar{h}_{j,\sigma}\right)
$$
which vanishes by Corollary \ref{final expression for the use of interpretation}.

As mentioned above, by the interpretation lemma the map

$$
B_{p}^{(i-1)} \rightarrow F[\theta_{1},\dots,\theta_{n}]\otimes B_{p}^{(i-1)}/(\theta_{k}\otimes q-1\otimes \delta_{k}^{\bar{b}}(q)|q\in \mathcal{S}_{p}^{(i-1)},\, i=1,\ldots,n)
$$
is an embedding and hence $\bar{f}=0$ in $B_{p}^{(i-1)}$. This completes the proof of Lemma \ref{The-following-holds(1)}.
\end{proof}

Thus we have a map $\phi: \mathcal{W} \rightarrow B_{p}$ where
$B_{p}$ is a representable algebra, $\Id(B_{p}) \supseteq \Id(W)$
and such that the space $\mathcal{S}_{p}$ is mapped
isomorphically. Consequently we have the following corollary.

\begin{corollary}\label{Kemer polynomials on the representable algebra}
Let $f$ be any Kemer polynomial of the algebra $W$ (at least $\mu+1$ small sets). Then $f \notin \Id(B_{p})$.

\end{corollary}
\begin{proof}
Since $\mathcal{W}_{0}$ is an affine relatively free algebra of
$W$, there exits an evaluation of $f$ on $\mathcal{W}_{0}$ which
is not zero. It follows that $f$ has a nonzero \textit{admissible}
evaluation $\bar{f}$ on $\mathcal{W}$ with $\bar{f} \in S_{p}$ and
hence $\bar{f} \notin ker(\phi)$. This proves the Corollary.

\end{proof}

\section{Representability - the proof}

We have all ingredients needed to prove the main theorem.
\begin{proof}
The proof is by induction on the Kemer index $p$ associated to a
$T$-ideal $\Gamma$ (containing a Capelli polynomial). If $p=(0,0)$
then $\Gamma=F\left\langle X\right\rangle $ and so $W=0$. Suppose
the theorem is true for any affine algebra with Kemer index
smaller than $p$. Denote by $S_{p}$ the $T$-ideal generated by all
Kemer polynomials corresponding to $\Gamma$, and let
$\Gamma'=\Gamma+S_{p}$. It is clear that the Kemer index of
$\Gamma'$ is strictly smaller than $p$. Hence, by the inductive
hypothesis there is a representable algebra $A'$ having $\Gamma'$
as its $T$-ideal of identities.

Let $B_{p}$ be the representable algebra constructed in
the previous section. We'll show $\Gamma=\Id(A'\times
B_{p})$.

It is clear that $\Gamma\subset \Id(A'\times B_{p})$ since $\Gamma$
is contained in $\Gamma'$ and by construction $\Gamma \subseteq
\Id(B_{p})$. Suppose there is $f\notin\Gamma$ with $f \in
\Id(A'\times B_{p})=\Id(A')\cap \Id(B_{p})$. Since $f\in \Id(A')=\Gamma'$, we may assume $f\in
S_{p}$. Using the Phoenix property, see Theorem \ref{Phoenix-property}, we
obtain a Kemer polynomial $f'$ (with at least $\mu+1$ small sets) such that $f'\in(f)$.
But by Corollary \ref{Kemer polynomials on the representable algebra}, $f \notin \Id(B_{p})$ and this contradicts our previous assumption on $f$.
This completes the proof.

\end{proof}


\begin{thebibliography}{99}

\bibitem{Al-Bel} E. Aljadeff and A. Kanel-Belov, {\em Representability and Specht problem for $G$-graded algebras},
Adv. Math. {\bf 225}, {2010}, no. 5, 2391--2428.

\bibitem{Bei} K. I. Beidar, {\em On A. I. Mal'tsev's theorems on matrix
representations of algebras},(Russian) Uspekhi Mat. Nauk. {\bf 41}
(1986), no. 5(251), 161--162.

\bibitem{BR} A. Kanel-Belov and  L.H. Rowen, {\em Computational Aspects of Polynomial Identities},
A. K. Peters Ltd., Wellesley, MA. (2005).

\bibitem{Kemer1.5} A.R. Kemer {\em Finite basis property of identities of associative
algebras}, Algebra i Logika {\bf 26} (1987), 597--641; English transl., Soviet Math. Dokl. {\bf 37} (1988), 362--397.

\bibitem{Kemer2} A. R. Kemer, {\em Ideals of Identities of
Associative Algebras}, Amer. Math. Soc., Translations of
Monographs {\bf 87} (1991).

\bibitem{Sviridova1} I. Sviridova, {\em Identities of PI-algebras graded by a finite abelian group}, Comm.
Algebra 39 (9) (2011) 3462--3490.

\bibitem{Sviridova2} I. Sviridova, {\em Finitely generated algebras with involution and their identities},
J. Algebra 383 (2013), 144--167.

\end{thebibliography}
\end{document}